\documentclass{article}

\usepackage{xr}
\usepackage{amsmath}
\usepackage{graphicx,psfrag,epsf}
\usepackage{enumerate}
\usepackage[square,numbers]{natbib}
\usepackage{mhequ}
\usepackage{datetime}
\usepackage{xfrac}

\usepackage{amssymb, amsthm,mathabx}
\usepackage{graphicx}
\usepackage{placeins}
\usepackage{url}
\usepackage{enumerate}

\usepackage{times}
\usepackage{bm}
\usepackage{natbib}
\usepackage{color}

\usepackage{bbm}

\newcommand{\blind}{1}

\def\spacingset#1{\renewcommand{\baselinestretch}%
{#1}\small\normalsize} \spacingset{1}

\theoremstyle{plain}

\newtheorem{theorem}{Theorem}[section]

\newtheorem{corollary}[theorem]{Corollary}

\newtheorem{lemma}{Lemma}[section]

\theoremstyle{definition}
\newtheorem{definition}[theorem]{Definition}

\def\T{{ \mathrm{\scriptscriptstyle T} }}

\newcommand{\bb}[1]{\mathbb{#1}}

\newcommand{\bigO}[1]{\mathcal{O}\left( #1 \right)}

\newcommand{\mc}[1]{\mathcal{#1}}
\DeclareMathOperator{\Nor}{No}
\newcommand{\No}[2]{\Nor\left( #1,#2 \right)}

\def \be{\begin{equs}}
\def \ee{\end{equs}}

\DeclareMathOperator{\correlation}{cor}
\DeclareMathOperator{\pr}{\bb P}
\DeclareMathOperator{\var}{var}

\newcommand{\PR}[1]{\bb P \left( #1 \right)}
\newcommand{\cor}[2]{\correlation\left[ #1,#2 \right]}
\newcommand{\Teff}{T_e}

\newcommand{\E}{\bb E}
\newcommand{\wh}[1]{\widehat{#1}}

\renewcommand{\P}{\mathcal P}
\DeclareMathOperator{\cov}{cov}

\DeclareMathOperator{\Binom}{Binomial}

\newcommand{\beginsupplement}{%
	\setcounter{table}{0}
	\renewcommand{\thetable}{S\arabic{table}}%
	\setcounter{figure}{0}
	\renewcommand{\thefigure}{S\arabic{figure}}%
	\setcounter{section}{0}
	\renewcommand{\thesection}{S\arabic{section}} 
}

\graphicspath{{./}}
\makeatletter
\def\input@path{{./}}
\makeatother

\begin{document}
\allowdisplaybreaks


\if1\blind
{
  \title{\bf MCMC for Imbalanced Categorical Data}
  \author{James E. Johndrow \thanks{
    J. Johndrow and D. Dunson acknowledge support from the United States National Science Foundation award number 1546130 and MaxPoint Interactive} \hspace{.2cm} \\
    Statistics Department, Stanford University \\
    Aaron Smith\thanks{A. Smith acknowledges support from the Natural Sciences and Engineering Research Council of Canada.} \\
    Department of Mathematics and Statistics, University of Ottawa\\
    Natesh Pillai \\
    Department of Statistics, Harvard Univeristy \\
    David B. Dunson\\
    Department of Statistical Science, Duke University
    }
    
  \maketitle
} \fi

\if0\blind
{
  \bigskip
  \bigskip
  \bigskip
  \begin{center}
    {\LARGE\bf Title}
\end{center}
  \medskip
} \fi

\begin{abstract}
Many modern applications collect highly imbalanced categorical data, with some categories relatively rare. Bayesian hierarchical models combat data sparsity by borrowing information, while also quantifying uncertainty.  However, posterior 
computation presents a fundamental barrier to routine use; a single class of algorithms does not work well in all settings and practitioners waste time trying different types of MCMC approaches.  This article was motivated by an application to quantitative advertising in which we encountered extremely poor computational performance for common data augmentation MCMC algorithms but obtained excellent performance for adaptive Metropolis.  To obtain a deeper understanding of this behavior, we give strong theory results on computational complexity in an infinitely imbalanced asymptotic regime.  Our results show computational complexity of Metropolis is logarithmic in sample size, while data augmentation is polynomial in sample size.  The root cause of poor performance of data augmentation is a discrepancy between the rates at which the target density and MCMC step sizes concentrate.   In general, MCMC algorithms that have a similar discrepancy will fail in large samples - a result with substantial practical impact.   
\end{abstract}

\noindent%
\vfill

\newpage
\spacingset{1.45} 

\section{Introduction}

It has become common to collect very large data sets, but in many settings, there is limited information in the data about many of the unknowns of interest, particularly when data are sparse and imbalanced.   Bayesian approaches are useful for borrowing information and characterizing uncertainty in these settings, but a fundamental barrier to routine implementation is posterior computation.  In general, when we are faced with an applied problem involving Bayesian modeling of complex data, the most time consuming and challenging stage of the implementation is not the choice of the model or priors but the `design' of the MCMC algorithm for posterior computation.  MCMC design \cite{gamerman2006markov, robert2013monte} remains more of an art than a science \cite{van2001art}, with expert Bayesian modelers using their substantial experience in choosing different types of algorithms, and combinations of algorithms, targeted to each new situation.  Although there are a variety of software packages for routine Bayesian computation in broad model classes -- for example Stan \cite{carpenter2016stan} and R-INLA \cite{rue2009approximate} -- such packages often do not work well in large and complex settings.  Bayesian researchers continue to spend substantial time trying out many different types of algorithms before (hopefully) finding ones that work well in a particular setting.

The over-arching goal of this article is to take a step in the direction of improving fundamental understanding of the contexts in which a particular type of MCMC algorithm should work well or not, allowing one to limit the need for trial and error.  Given the sparsity of the relevant literature, we are necessarily quite modest in the scope of problems we focus on, but nonetheless obtain what we feel is a broadly useful result that should help practitioners to take more of a scientific approach to MCMC design.  To formally assess whether an algorithm `works well' we use the lens of computational complexity theory. The goal of MCMC is to obtain samples from the posterior for use in constructing statistical estimators of posterior summaries of interest; we would like these estimators to have low mean squared error even if we have run our algorithms for a limited clock time.  In general, computational efficiency depends on time per iteration and the mixing rate of the Markov chain.  As the `problem size' increases, both of these factors tend to slow down; computational complexity theory describes the rate of this slow down. If the rate is too high, then the MCMC algorithm may be practically useless in `big' problems. Problem size is a general term but may correspond to the sample size, data dimensionality, parameter dimensionality or other aspects measuring the hardness of the problem. 

Most of the existing literature studying the efficiency of MCMC algorithms has focused on showing the Markov chain mixes well in the sense of being geometrically or uniformly ergodic \cite{meyn2012markov, meyn1994computable, roberts1997geometric, rosenthal1995minorization}; for example, refer to \citet{choi2013polya} and \citet{roy2007convergence}, which show such conditions for two of the algorithms studied here.\footnote{A parallel and interesting literature exists on optimal scaling, see e.g. \cite{roberts1997weak}.}  However, we find that such results tell a practitioner very little about whether the algorithm works well in a particular context or not; some simple examples where bounds of this type are very loose are given in \cite{diaconis2008gibbs}.  
Part of this is due to the fact that the constants in the bounds often depend critically on the problem size in a way that is inconsistent with empirical performance \cite{rajaratnam2015mcmc}. Therefore, to study computational compexity, it is necessary to obtain bounds that are sharp as a function of problem size, which is considerably more difficult; \cite{hairer2014spectral} is a prominent (successful) example.  See \cite{belloni2009computational,mossel2006limitations,yang2016computational} for some precedents applying computational complexity theory to MCMC. Most of these studies focus on a single model and MCMC algorithm and show either upper or lower bounds, whereas here we seek to compare different types of algorithms for the same model. 

We are particularly interested in models for categorical data.  In such settings, it is routine to rely on data augmentation to simplify design of MCMC algorithms \cite{tanner1987calculation, damien1999gibbs}.  An amazing variety of clever schemes have been introduced so that one can sample from simple conditional distributions for the parameters after introducing latent data \cite{holmes2006bayesian, fruhwirth2010data, polson2013bayesian, albert1993bayesian}.  Key examples include the Gaussian data augmentation (DA) scheme for probit models of \citet{albert1993bayesian} and the Polya-gamma DA approach for logistic regression of \citet{polson2013bayesian}.  We and many others routinely use these algorithms in all sorts of applied contexts, and they are often remarkably successful.  However, sometimes they fail dramatically for unknown reasons, producing very poor mixing.  In such cases, one can instead avoid introducing latent data, and use more generic Metropolis, adaptive Metropolis or Hamiltonian Monte Carlo (HMC) methods.

Two examples we have encountered include computational advertising (see Section \ref{sec:Application}) and ecological modeling of biological communities (for example, \cite{ovaskainen2015using}).  In these cases, we wasted months developing code, error checking and refining DA-MCMC algorithms before shifting focus to other types of approaches.  However, we noticed a commonality in these problematic applications for DA-MCMC: both involved large and very imbalanced data, with some events or species being rare.  We have found that this behavior occurs routinely, essentially regardless of the type and complexity of the statistical model, if the data are large and imbalanced.  To obtain insight into why this occurs with a goal of providing guidance to practitioners, we carefully study computational complexity of DA-MCMC and Metropolis algorithms under an infinitely imbalanced asymptotic regime introduced by \citet{owen2007infinitely} in studying estimation performance in logistic regression.  

Although our theory is focused on a simple case, the technical details are very far from straightforward, and the results lead to substantial new insight, which has already led to new algorithms \cite{duan2017calibrated}.  In particular, we find that the root case for the poor behavior of DA-MCMC is a discrepancy between the rates at which the target density and MCMC step sizes concentrate.  Such a discrepancy will lead to poor scaling of MCMC algorithms in other contexts as well, and to our knowledge we are the first to notice this. This insight is possible because we consider a non-standard asymptotic framework that more accurately approximates the properties of the posterior in finite samples. An important implication is to avoid DA in large imbalanced data contexts.  

Section 2 motivates the problem and describes the practical behavior of various MCMC algorithms through our computational advertising application.  Section 3 contains our main theoretical results, while providing an intuition.  Section 4 shows that the predictions from our theory hold in broad imbalanced data applications, and Section 5 contains a discussion.  Proofs are included in an Appendix.

\section{Motivating application} \label{sec:Application}
This article was motivated by an application to quantitative advertising. Advertisers seek to optimize the yield or click through rate for display advertisements. There are thousands of websites \emph{serving} ads -- showing display ads for a fee -- and advertisers must bid on these \emph{impressions} -- placements of an ad for their client's website in a particular location on a site serving ads -- in auctions that take place in a fraction of a second when a user navigates to the serving site. Advertisers develop models of the value of showing a particular advertisement to a user given features on the user, serving site, and the site being advertised. 

An important component of these models is the estimated probability $p_i$ that a user visits serving site $i$ and the client site in the same browsing session. The idea is that if visitors to site $i$ tend to be more interested in the client's products than visitors to most other serving sites even without being shown an ad, then a visitor to site $i$ will be more likely to click on an advertisement for the client's product(s). Let $n_i$ be the total number of visitors to serving site $i$ in some study period, and $y_i$ be the number of users who also visited the client's site in the same browsing period, with $i=1,\ldots,N$. The $y_i$ tend to be small -- in many cases, $y_i \le 10$ -- and the empirical probabilities $\frac{y_i}{n_i}$ on the order of $10^{-3}$ to $10^{-5}$. A histogram of logit transformed $\frac{y_i+1}{n_i}$ for the motivating dataset we obtained from MaxPoint Interactive is shown in Figure \ref{fig:EmpiricalProb}. Also shown is a histogram of $\log(y_i+1)$. The data are about 74\% sparse, and of the nonzero observations of $y_i$, the 25th percentile is 7 and the median is 13. 

\begin{figure}
\centering
\includegraphics[width=0.8\textwidth]{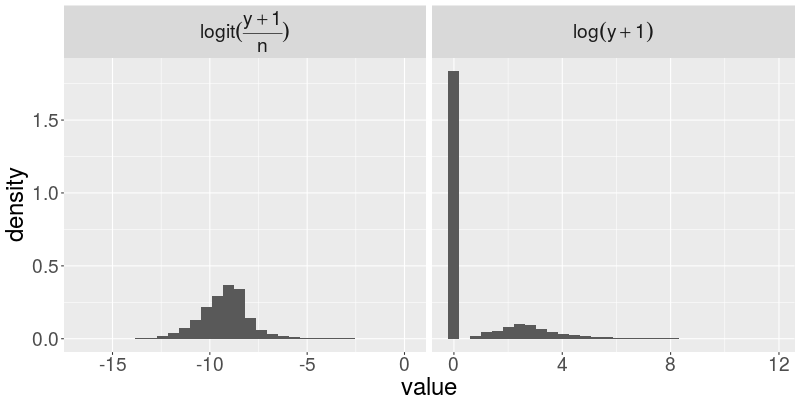}
\caption{Histograms of co-browsing data as described in text.} \label{fig:EmpiricalProb}
\end{figure}

A natural Bayesian approach to obtaining low-risk estimates of the $p_i$ is to borrow information across the serving sites via a hierarchical model:
\be \label{eq:HierModel}
y_i \mid n_i, p_i &\sim \Binom\left( n_i, g^{-1}(\theta_i) \right), \quad \theta_i \stackrel{iid}{\sim} \Nor(\theta_0,\sigma^2) \\
\theta_0 &\sim \Nor(b,B), \quad \sigma \sim \pi(\sigma),
\ee
where $\pi(\sigma)$ is a half Cauchy prior on $\sigma$, recommended as a prior on variance components in hierarchical models by \citet{gelman2006prior} and \citet{polson2012half}. In this application, we use a moderately informative prior of $b=-12$ and $B = 36$, consistent with Figure \ref{fig:EmpiricalProb}, though the results that follow were insensitive to the prior choice. 

Hierarchical generalized linear models are commonly estimated using data augmentation Gibbs samplers of the form
\be \label{eq:daupdate}
\omega \mid \theta, y &\sim p(\omega \mid \theta, y)  \\
\theta \mid \omega, y &\sim \mathrm{No}( \mu(\omega), \Sigma(\omega) ).
\ee
where $p(\omega \mid \theta,y)$ is a \emph{P\'{o}lya-Gamma} distribution when $g^{-1}$ is the inverse logit link, and truncated Gaussian when $g^{-1}$ is the inverse probit link.  Applying this approach to the MaxPoint data in the logistic case, the sampler had remarkably poor efficiency. Running the MCMC sampler for 50,000 iterations, with 30,000 iterations discarded as burn-in, the right panel of Figure \ref{fig:acfhier} shows the empirical autocorrelation function up to lag 100 for the $N=59,317$ different $\theta_i$ parameters.  Even at lag 100, the autocorrelations were high.  In contrast, as shown in the left panel, a simple adaptive Metropolis algorithm, having less computational time per iteration, had dramatically better mixing.

\begin{figure}
\centering
\includegraphics[width=0.8\textwidth]{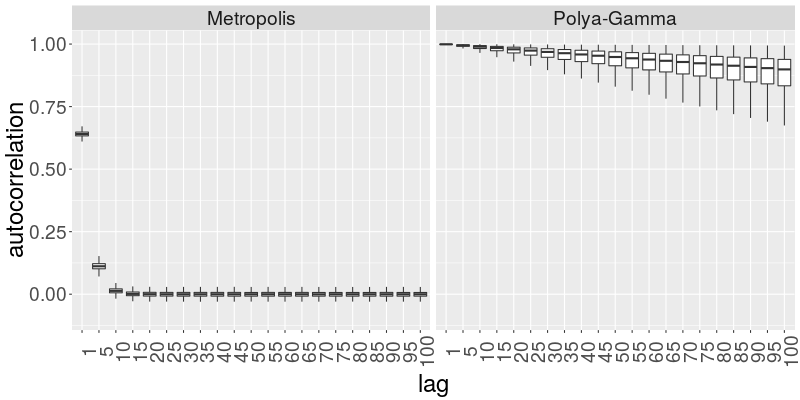}
\caption{Boxplots of estimated autocorrelations for the $\theta_i$ parameters.  Left panel: adaptive Metropolis. Right panel: Polya-gamma data augmentation.  Outliers suppressed for readability.}
 \label{fig:acfhier}
\end{figure}

The much greater efficiency of Metropolis compared with data augmentation persists in the probit case and well beyond this particular setting -- even when there is no hierarchical structure and we have only one site, so long as $y \ll n$.  Figure \ref{fig:acfsimple} shows estimated autocorrelations for data augmentation and Metropolis with logit link for $y=1$ with increasing $n$, i.e.
\be \label{eq:InterceptModel}
y \mid n, \theta \sim \Binom( n,g^{-1}(\theta)), \quad \theta \sim \Nor(0,B)
\ee
with $B=100$. For data augmentation, the autocorrelations increase markedly with $n$, while for Metropolis, the autocorrelations are insensitive to $n$. We have found this behavior to be unrelated to centering the prior on $\theta$ at zero and the choice of $B$. In this case, Metropolis uses a Gaussian random walk proposal with unit variance. 

At least part of the phenomenon observed in this application is a generic feature of data augmentation when the data are imbalanced.  The remainder of this paper aims to demonstrate theoretically why this occurs, showing that it is a general phenomenon for data augmentation samplers in the large sample, highly imbalanced data settings that dominate applications ranging from modeling human behavior in industry settings to ecology.  Our results are of direct interest to practitioners conducting applied Bayesian modeling in these settings, while producing important insights into factors controlling efficiency of augmentation-based MCMC algorithms in broad settings.

\begin{figure}
\centering
\includegraphics[width=0.8\textwidth]{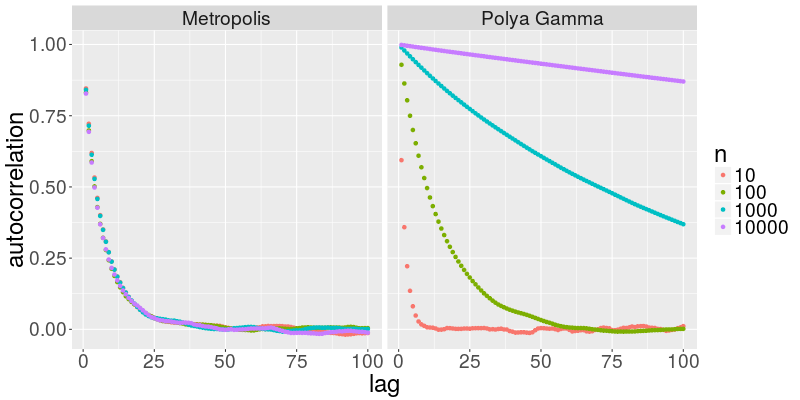} 
\caption{Plot of estimated autocorrelations in the intercept-only case as described in text.} \label{fig:acfsimple}
\end{figure}

\section{Theory results} \label{sec:theory}
In this section, we give a brief overview of the computational complexity of MCMC and its relationship to the \emph{spectral gap} and \emph{conductance} of the associated transition kernel. We then give theoretical results that are consistent with the empirical performance of data augmentation and Metropolis for imbalanced data. Our results are given for the intercept-only case with $y=1$ and increasing $n$, as in Figure \ref{fig:acfsimple}. \citet{owen2007infinitely} referred to this as the \emph{infinitely imbalanced} regime. We focus on this simple case because it reflects the high level of imbalance that we observe in the application, the poor performance of DA in the hierarchical model persists in the intercept only case shown in Figure \ref{fig:acfsimple}, and obtaining bounds that are sufficiently sharp to have relevance to the practical performance of these algorithms is highly non-trivial even in this simple setting. We later show empirically that this behavior is found in a wide variety of data augmentation algorithms for binomial and multinomial likelihoods, including for regression applications, when the data are imbalanced. 

\subsection{Computational complexity of MCMC}
We assume some facility with concepts such as the integrated autocorrelation
time, spectral gap, mixing time, and conductance of a Markov chain, as well
as Gibbs sampling and the Metropolis-Hastings algorithm. Please refer to 
the Supplement for a primer of these concepts. 

Let $\{\Theta_k\}$ be a Markov chain with transition kernel $\mc P$ on 
a Polish state space $\mathcal T$ having invariant measure $\Pi$. 
In our application, $\mathcal T = \bb R^p$. The goal of MCMC is to 
approximate the expectation of functions $f : \mathcal T \to \bb R$
under $\Pi$ by Ces\`{a}ro averages
\be \label{eq:ergavg}
\Pi f = \int_{\mathcal T} f(\theta) \Pi(d\theta \mid y) \approx \frac{1}{T} \sum_{k=0}^{T-1} f(\theta_k) \equiv \widehat{f}_T,
\ee
with $\theta_0,\theta_1,\ldots,\theta_{T-1}$ a single realization of
the Markov chain. A common measure of performance is the MCMC 
mean squared error
\be
\Delta(\Pi f, \widehat f_T) = \bb E \left( \Pi f - \frac1T \sum_{k=0}^{T-1} f(\Theta_k) \right)^2 &= \left( \Pi f - \frac1T \sum_{k=0}^{T-1} \nu \mc P^k f \right)^2 + \frac{1}{T^2} \sum_{j=0}^{T-1} \sum_{k=0}^{T-1} \cov(f(\Theta_j),f(\Theta_k)) \\
&= \text{Bias}^2 + \text{Variance}, \label{eq:MCMCMSE}
\ee
where $\Theta_0 \sim \nu$ and the expectation is with respect to the 
law of $\Theta_0,\ldots,\Theta_{T-1}$. If one can compute
$\Delta$ as a function of $n$ 
and the computational complexity of one step from $\mc P$
is known, then it is natural to measure overall computational complexity
by multiplying these two factors (see e.g. \cite{johndrow2015approximations}). 
For instance, if $\Delta$
converges to $\infty$ at the rate $n^a$ and one step from
$\mc P$ costs $n^b$, then the overall computational complexity
is $n^{a+b}$ up to constants.

The variance term in \eqref{eq:MCMCMSE} motivates
empirical analysis of the performance of the algorithm through estimates
of the autocorrelations $\rho_k$ at lag $k$. Another common empirical 
performance metric is the \emph{effective sample size} $T_e$, which
is roughly proportional to $1/\Delta(\Pi f, \widehat f_T)$. Informally,
$T_e$ is the number of independent samples from $\Pi f$ that would
give variance equal to a path of length $T$ from $\mathcal P$
with $\Theta_0 \sim \Pi$. For interpretability, it is useful to compute
$T_e/T$ or $T_e/t$, where $t$ is total computation (wall clock) time. 

The mean squared error $\Delta$ can also be analyzed theoretically.
For reversible $\mathcal P$, one can obtain both asymptotic (in $T$) and 
finite $T$ bounds on $\Delta$ in terms of the $L^2(\Pi)$ \emph{spectral gap}
$\delta(\mathcal P)$ of $\mathcal P$, with the leading terms order
$1/(\delta T)$ (see \cite{jones2004markov} for asymptotic bounds and 
\cite{rudolf2011explicit,paulin2015concentration} for 
finite-time bounds, among others). 
These bounds are generally given for the
supremum over all $f \in L^2(\Pi)$. 
The asymptotic bound is sharp for worst case functions.
Transition kernels defined by Metropolis 
algorithms are in general reversible. Although the transition kernels
defined by data augmentation Gibbs samplers are not reversible,
the marginal chain for $\theta$ is reversible. Our results for
DA will pertain to the $\theta$-marginal chain.


We use two strategies to obtain bounds on $\delta(\mathcal P)$. The
first is to obtain lower bounds by 
a drift and minorization argument in the style of 
\cite{rosenthal1995minorization}.
The second is to obtain upper bounds by upper bounding the 
\emph{conductance} or \emph{Cheeger
constant} $\kappa$ of $\mathcal P$, and then employing 
the inequality \cite{lawler1988bounds}
\be \label{eq:KappaDelta}
\frac{\kappa^2}{8} \le \delta(\mathcal P) \le \kappa.
\ee
The conductance also gives bounds on the mixing time of $\mathcal P$
when $\nu$ satisfies a ``warm start'' condition. The resulting upper bounds
on mixing times are approximately order $\kappa^{-2}$ (see the Supplement).
By obtaining bounds on $\delta$ or $\kappa$ and studying the rate at which
these bounds converge to zero as $n \to \infty$, we obtain estimates of the
computational complexity of the algorithm. 


\subsection{Main results}
We now give bounds on $\delta$ for a Metropolis algorithm as well as
data augmentation algorithms for logit and probit. The Metropolis result
gives a \emph{lower bound} on the spectral gap, showing that the computational
complexity of the algorithm cannot be worse than $(\log n)^3$. The
results for the data augmentation algorithms give an \emph{upper bound}
on the spectral gap of order $n^{-1/2} (\log n)^k$, with $k=2.5 or 5.5$, 
depending on the algorithm. 
Since each step of the
data augmentation sampler requires sampling $n$ auxiliary variables,
this suggests the computational
complexity is order $n^{3/2} (\log n)^k$. 

In the results that follow, we use the notation $f(n) \gtrsim g(n)$ (or $f(n) \lesssim g(n)$)
to mean 
there exists a constant $C$ and $n_0<\infty$ such that for all $n>n_0$, $f(n) > C g(n)$ 
(or $f(n) < C g(n)$, respectively).
\begin{theorem} \label{thm:mh}
Let $\mc P$ be the transition kernel of a Metropolis algorithm for the 
model in \eqref{eq:InterceptModel} with $y=1$ and proposal kernel
$q(\theta,\cdot) = \text{Uniform}(\theta-\log n, \theta+\log n)$. 
Then 
\be
\delta_n(\mc P) \gtrsim  (\log n)^{-3}.
\ee
\end{theorem}
Since the Metropolis algorithm has cost per step that is independent
of $n$, this immediately implies that the computational complexity of 
the algorithm is at worst $(\log n)^3$ via the upper bounds on 
$\Delta(\widehat f_T, \Pi f)$ of order $1/(\delta T)$. 

The proof proceeds by showing a Lyapunov function and a minorization condition, 
then applying
\cite[Theorem 5]{rosenthal1995minorization}. The full proof is given in the 
Appendix, but we highlight an aspect of the argument that is to our knowledge
unusual in a proof of this type and should be useful in proving drift conditions
for Metropolis and Metropolis-Hastings algorithms generally. 

When sampling
from a target density $p(\theta)$, one always has the relationship
\be
p'(\theta) = \left( \frac{d}{d\theta} \log p(\theta) \right) p(\theta).
\ee 
Suppose $z(\theta) >  \frac{d}{d\theta} \log p(\theta)$ for all $\theta$.
Then $p'(\theta) \le z(\theta) p(\theta)$ and by Gr\"{o}nwall's 
inequality
\be
p(\theta) \le p(\theta_0) \exp \left( \int_{\theta_0}^{\theta} z(\theta) d\theta \right).
\ee
The usefulness of this strategy is that to control Metropolis acceptance probabilities,
one needs a bound on the ratio $\frac{p(\theta)}{p(\theta_0)}$. Often, it is easier to bound 
$\frac{d}{d\theta} \log p(\theta)$
than it is to bound $\frac{p(\theta)}{p(\theta_0)}$ directly. In a Bayesian model, 
$p(\theta) \propto \log(L(y \mid \theta)) 
+ \log(\pi(\theta))$, the sum of the log prior and log likelihood. In our setting, applying
the mean value theorem to $\frac{d}{d\theta} \log p(\theta)$ gives us a bounding function
$z(\theta) \le \frac{-\theta}{B}$, and we obtain exponentially decaying acceptance 
probabilities as we move away from the mode. This allows us to show that 
$V(\theta) = \exp(-|\hat{\theta}-\theta|)$ is a Lyapunov function for $\mathcal P$, where
$\hat{\theta}$ is the posterior mode. We expect this strategy to be useful for constructing
good Lyapunov functions for Metropolis and Metropolis-Hastings algorithms for unimodal
targets. In multiple dimensions, Gr\"{o}nwall-Bellman like inequalities for multivariate
differential equations may allow extension of this approach.

The next theorem gives an upper bound for the spectral gap for data augmentation for
logit or probit models.  
\begin{theorem} \label{thm:pgcond}
Let $\mc P$ be the transition kernel of the P\'{o}lya-Gamma data augmentation sampler
for the model in \eqref{eq:InterceptModel} with $y=1$ and $g^{-1}$ the inverse logit function. 
Then the conductance of $\mc P$ satisfies
\begin{align} \label{IneqPgcondMain}
\kappa_n(\mc P) \lesssim \frac{ (\log n)^{5.5}}{\sqrt{n}}.
\end{align}

Similarly, if $\mc P$ is the transition kernel of the Albert and Chib data augmentation sampler
for the model in \eqref{eq:InterceptModel} with $y=1$ and $g^{-1}$ the inverse probit link. Then
the conductance of $\mc P$ satisfies
\be \label{IneqAccondMain}
\kappa_n(\mc P) \lesssim \frac{(\log n)^{2.5}}{\sqrt{n}}.
\ee
\end{theorem}
By \eqref{eq:KappaDelta}, this immediately gives asymptotic (in $n$) upper bounds on
the spectral gap $\delta_n(\mathcal P)$ for both algorithms.  Discarding log factors, these bounds show that the spectral gap converges to zero in the square root of sample size or faster, explaining the results in Section 2 showing very poor mixing that got worse and worse with increasing $n$.

To obtain a bound on computational complexity, one must factor in how the time per iteration increases with $n$.  This suggests a computational complexity of order between 
$n^{3/2} (\log n)^c$ and $n^{2} (\log n)^c$ for both algorithms, with $c \le 5.5$. These estimates are 
obtained by applying \eqref{eq:KappaDelta} to the bounds in Theorem \ref{thm:pgcond},
then multiplying by a factor of $n$ for the linear complexity of one step from 
either of these kernels. In the Supplement, we provide empirical evidence that the 
computational complexity is approximately $n^{1.85}$ based on multiple very long 
runs of both algorithms. 


\subsection{Intuition} \label{sec:intuition}
The key cause of order $n^{-1/2}$ conductance for data augmentation for highly imbalanced data is a discrepancy between the width of the high-probability region of the posterior as a function of $n$ and the rate at which the step sizes for these algorithms converge to zero. We give a rough characterization of this phenomenon for P\'{o}lya-Gamma; probit data augmentation is similar. 

When $\theta_t$ is close to the posterior mode, the mean and variance of $\omega$ satisfy
\be
\bb E[\omega_t \mid \theta_t] \approx \frac{n}{\log(n)}, \quad \var[\omega_t \mid \theta_t] \approx \frac{n}{(\log(n))^3},
\ee
so by Chebyshev's inequality, $\omega_{t+1}$ is in the interval $\frac{n}{\log(n)} \pm \sqrt{n} (\log(n))^{-3/2}$ with high probability for large $n$. Also, 
\be 
\bb E[\theta_{t+1} \mid \omega_{t+1}] \approx \frac{-n}{\omega_{t+1}}, \quad \var[\theta_{t+1} \mid \omega_{t+1}] \approx \omega^{-1},
\ee 
and $\theta_{t+1} \mid \omega_{t+1}$ is conditionally Gaussian, so with probability converging to 1 exponentially fast, $\theta_{t+1}$ is in the interval $-\log(n) \pm (\log n)^{3/2} n^{-1/2}$ -- a ball of radius $(\log n )^{3/2} n^{-1/2}$ around the posterior mode, which is approximately $-\log(n)$. This concentration of measure phenomenon means that step sizes in the bulk of the posterior are not much larger than $n^{-1/2}$ with high probability.

On the other hand, As $n \to \infty$, the bulk of the posterior has width at least $(\log n)^{-1}$, \emph{so it is not contracting at the same rate as the step sizes}.  The infinitely imbalanced regime is a non-standard asymptotic setup designed to reflect the extreme imbalance observed in many modern categorical data applications.  In classical statistical asymptotics, the samples size $n \to \infty$ for data of fixed dimension generated from a likelihood with a \emph{fixed true parameter value}.  To mimic high dimensional data applications, it has become popular to study the regime where the dimension of the data and/or parameter $\theta$ to diverge with $n$; the infinitely imbalanced setting instead fixes the number of successes $y$ to effectively drive the true parameter to zero as $n \to \infty$.  This leads the posterior to concentrate around the mode at a rate no faster than $(\log n)^{-1}$, instead of the usual $n^{-1/2}$ rate, consistent with the empirical observation that substantial estimation uncertainty remains in highly imbalanced cases even when sample sizes are huge. We emphasize that this setup was selected \emph{because the data in our motivating application were extremely imbalanced}, and the theoretical results predict the observed performance, and not for theoretical convenience or to obtain any particular result. 

This mismatch between typical step sizes of $\mc P$ and the width of the posterior bulk results in the Markov chain becoming trapped too close to the mode, as illustrated by the graphic in Figure \ref{fig:cartoon}. This reasoning applies to essentially \emph{any MCMC algorithm}: if the width of the high probability region of the posterior and step sizes inside that region are of different order in $n$, then the algorithm will converge slowly and exhibit very high autocorrelations in large samples.

\begin{figure}[h]
\centering
\includegraphics[width=0.4\textwidth]{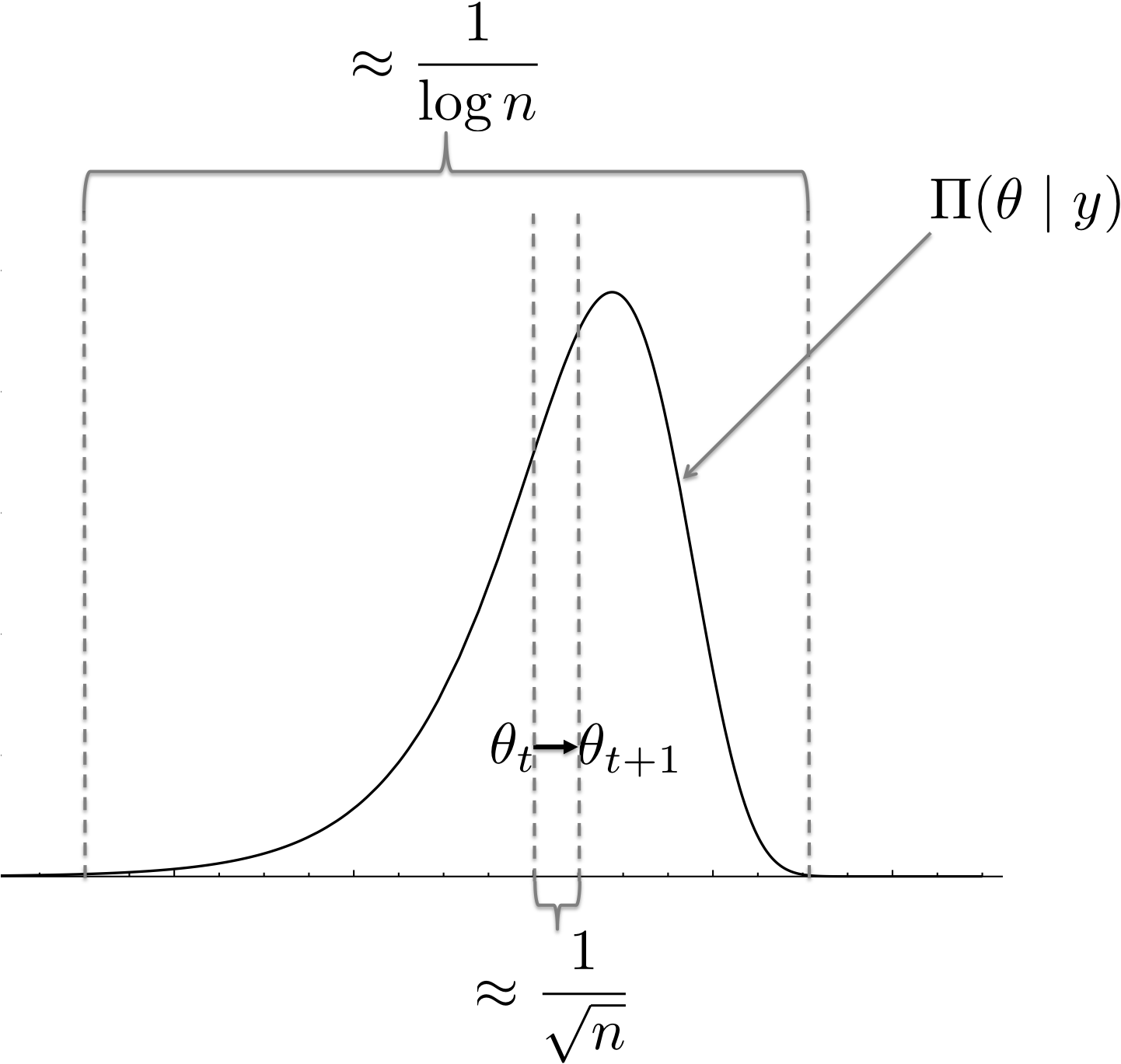}
\caption{Cartoon comparing the high posterior density region and typical move size. } \label{fig:cartoon}
\end{figure}

The success of Metropolis-Hastings in this case is easily explained. The typical step sizes of the kernel can be tuned through the choice of $q$. Since the posterior is contracting at a rate no faster than $(\log n)^{-1}$, one sets $q$ to propose moves larger than $(\log n)^{-1}$. In particular, the bound of $(\log n)^{-3}$ is composed of a factor of $(\log n)^{-1}$ from the return time bound, and a factor of $(\log n)^{-2}$ from an application of two-step minorization using a set of width $\log(n)$ outside of which the posterior is negligible and a set of width $(\log n)^{-1}$ containing the mode. The empirical analysis that follows shows little sensitivity of effective sample size in Metropolis-Hastings or Hamiltonian Monte Carlo algorithms to $n$.

\section{Empirical analysis of general imbalanced data applications}
In this section, we aim to show through simulation studies that poor mixing of data augmentation samplers occurs in many imbalanced data settings, including binomial regression models. In each case, alternative Metropolis algorithms perform much better. In the supplement, we conduct additional simulation studies suggesting that data augmentation algorithms for multinomial logit and probit have similar behavior when data are imbalanced. We conclude this section by returning to the original application. 

\subsection{Binomial logit and probit}
In the first set of examples, we consider the model in \eqref{eq:InterceptModel} with probit and logit link. 
We set $y=1$ and vary $n$ between $10$ and $10,000$. We perform computation using the Albert and Chib data augmentation algorithm for the probit link and the P\'{o}lya-Gamma data augmentation algorithm for the logit link, then estimate autocorrelations and effective sample sizes. 
In each case we use a prior of $b=0, B=100$. 
For probit, we use the implementation in the \texttt{bayesm} package for R. For logit, we use the package \texttt{BayesLogit}. For comparison, we implement random walk Metropolis with $q(\theta,\cdot) \sim \Nor(\theta,1)$ as proposal distribution for the model with logit link.
Table \ref{tab:effsize} shows $\Teff/t$ (rounded to the nearest integer), computed using the \texttt{coda} package for \texttt{R}. Effective samples per second is anemic for the data augmentation Gibbs samplers for large $n$, but largely insensitive to $n$ for random walk Metropolis.

\begin{table}[ht]
\centering
\begin{tabular}{rrrrr}
  \hline
 & n = 10 & n = 100 & n = 1000 & n = 10000 \\ 
  \hline
Albert and Chib & 16421 & 313 & 5 & 0 \\ 
  Polya-Gamma & 95989 & 1623 & 25 & 0 \\ 
  Metropolis & 5106 & 5389 & 5668 & 4922 \\ 
   \hline
\end{tabular}
\caption{$T_e/t$ for data augmentation and Metropolis algorithms
             with $y=1$ and varying $n$} 
\label{tab:effsize}
\end{table}

Although the theoretical results in Section \ref{sec:theory} consider the case where $y=1$ and $n$ is increasing, empirically we observe poor mixing whenever $y/n$ is small. To demonstrate this, we perform another set of computational examples where $y$ and $n$ both vary in such a way that $y/n$ is constant. Specifically, we consider $n=10,000$, $n=20,000$, and $n=50,000$ with $y=1, 2, 5$. Computation is performed for the two data augmentation Gibbs samplers as above, and effective sample sizes and autocorrelation functions estimated. Fig. \ref{fig:acfsygt1} shows estimated autocorrelations, which are similarly near 1 at lag 1 and decay slowly. Table \ref{tab:esygt1} shows values of $\Teff/t$ for the two algorithms. Neither measure of computational efficiency shows a meaningful effect of increasing $y$ when $y/n$ remains constant.
\begin{figure}[h]
\centering
\includegraphics[width=0.7\textwidth]{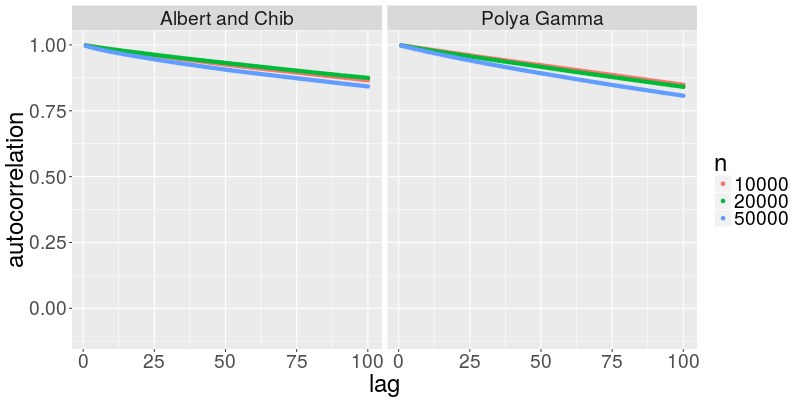}
\caption{Estimated autocorrelation functions for synthetic data examples that vary $y$ and $n$ with $y/n = 10^{-4}$ in each case.} \label{fig:acfsygt1}
\end{figure}

\begin{table}[ht]
\centering
\begin{tabular}{rrrr}
  \hline
 & n = 10000 & n = 20000 & n = 50000 \\ 
  \hline
Albert and Chib & 0 & 0 & 0 \\ 
  Polya-Gamma & 0 & 0 & 0 \\ 
  Metropolis & 4762 & 11220 & 16768 \\ 
   \hline
\end{tabular}
\caption{$T_e/t$ for data augmentation and Metropolis algorithms
             with varying $n$ and $y$ with $y/n=10^{-4}$ in each case.} 
\label{tab:esygt1}
\end{table}


\subsection{Binomial regression}
We now consider a binomial regression model of the form
\be \label{eq:BinomReg}
y_i \mid x_i, \beta &\sim \Binom(n_i,g^{-1}(x_i \beta)),\quad i=1,\ldots,N \quad \beta \sim \Nor(0,B).
\ee
We form the predictor matrix $X$ by putting $x_{i,1} = 1$ and sampling $x_{i,2:p} \sim \text{Uniform}(-1,1)$. We then simulate from \eqref{eq:BinomReg}
where $g^{-1}$ is the inverse logit link, with $\beta_1 = \alpha$ and $\beta_{2:p} \sim \mathrm{No}(0,1)$. We set $n_i = 1,000$ for all $i$, $N=1,000$, and consider $p=20$ and $p=100$. We vary $\alpha$ between $-5$ and $-10$, giving a series of increasingly imbalanced data settings. The means of $y_i$ are given in table \ref{tab:ymeansims}.
\begin{table}[ht]
\centering
\begin{tabular}{rrrrrrr}
  \hline
 & $\alpha=$-5 & $\alpha=$-6 & $\alpha=$-7 & $\alpha=$-8 & $\alpha=$-9 & $\alpha=$-10 \\ 
  \hline
p=20 & 98.83 & 62.27 & 37.63 & 21.86 & 11.98 & 6.00 \\ 
  p=100 & 172.91 & 131.16 & 96.87 & 69.83 & 50.36 & 35.78 \\ 
   \hline
\end{tabular}
\caption{mean of y in simulation study} 
\label{tab:ymeansims}
\end{table}

For each simulation, we perform computation using P\'{o}lya-Gamma data augmentation with the \texttt{BayesLogit} package. In this moderate dimension setting, construction of good Metropolis proposals can be challenging, so we use HMC implemented in \texttt{Stan}, which can be viewed as the use of simulated Hamiltonian dynamics to generate an efficient Metropolis proposal. Results summarized by estimated values of $\Teff/T$ and $\Teff/t$ are shown in Figure \ref{fig:binomreg}. 
In all cases, effective sample size and effective samples per second are orders of magnitude larger for HMC than for PG data augmentation. Additionally, HMC shows little sensitivity to the level of imbalance, while the performance of data augmentation degrades noticeably as the level of imbalance increases. 
\begin{figure}[h]
\centering
\begin{tabular}{cc}
$\Teff/T, p=20$ & $\Teff/t, p=20$ \\
\includegraphics[width=0.4\textwidth]{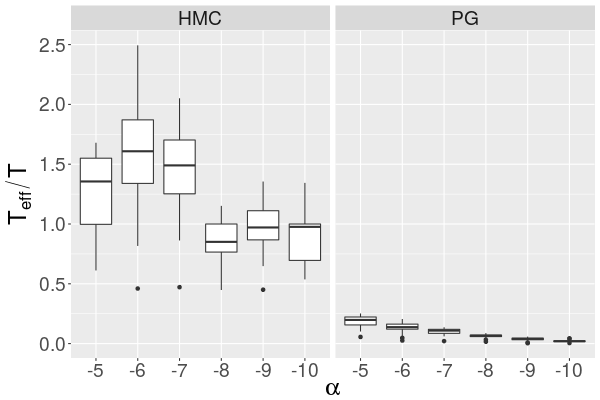} & \includegraphics[width=0.4\textwidth]{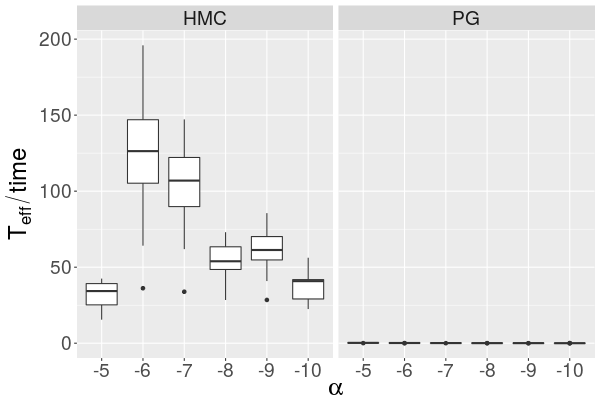} \\
$\Teff/T, p=100$ & $\Teff/t, p=100$ \\
\includegraphics[width=0.4\textwidth]{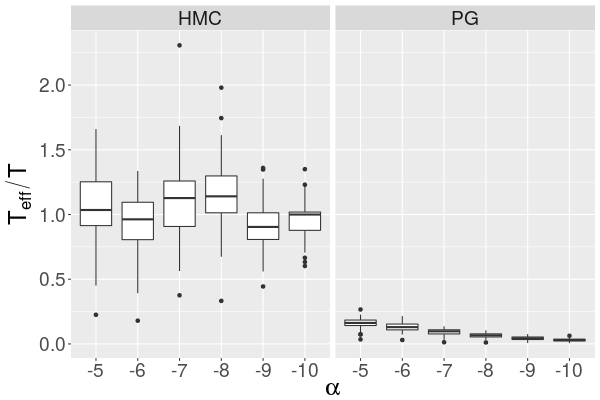} & \includegraphics[width=0.4\textwidth]{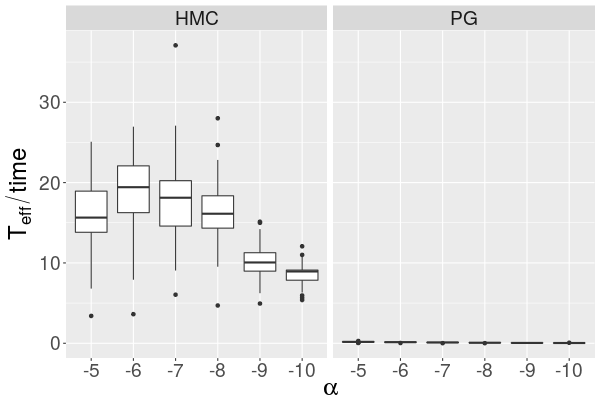} \\
\end{tabular}
\caption{$\Teff/T$ and $\Teff/t$ for general binomial regression examples for $p=20,100$ for logistic regression computation by P\'{o}lya-Gamma data augmentation and Hamiltonian Monte Carlo. Boxplots show distribution of the indicated quantity over the $p$ parameters of the model.} \label{fig:binomreg}  
\end{figure}

\subsection{Quantitative advertising reprise}
We now give details on the adaptive Metropolis algorithm we employed with success in the quantitative advertising application, and provide some additional results. Our alternative to data augmentation for the model in \eqref{eq:HierModel} with logit link has the update scheme
\be \label{eq:MetropolisUpdate}
\text{update }& (\{\theta_i\} \mid \theta_0, \sigma,y) \text{ for } i =1,\ldots,n  \text{ using Metropolis} \\
\text{update }& (\theta_0 \mid \sigma, \theta_1,\ldots,\theta_n) \text{ using Gibbs}  \\
\text{update }& (\sigma \mid \theta_0,\theta_1,\ldots,\theta_n) \text{ using slice sampling}.
\ee
The complete algorithm is given in the Appendix. We detail the Metropolis update here, which we construct using a variation of the adaptive Metropolis algorithm of \citet{haario2001adaptive}. We construct a time-inhomogeneous proposal $\theta_i^*$ for each component $\theta_i$ of $\theta$ using the proposal kernel
\be \label{eq:AdaptiveQ}
q_k(\theta_{ik},\theta^*_i) = \phi( (\theta_i^*-\theta_{ik})^2/\sigma_{ik} ), \quad \sigma_{ik} = \frac{s}{k} \sum_{j=0}^{k-1} (\theta_{ij}-\bar{\theta}_{ik})^2, 
\ee
where $\phi(\cdot)$ is the univariate standard Gaussian density and $\bar{\theta}_{ik}$ is the average of $\theta_{i0},\theta_{i1},\ldots,\theta_{i(k-1)}$, the first $k$ realizations of $\theta_i$. We then make independent Metropolis acceptance decisions for each component $i$; since the $\theta_i$ are all conditionally independent given $\theta_0,\sigma$, these updates are made in parallel. The original adaptive Metropolis update of \cite{haario2001adaptive} suggests making a joint proposal for $\theta$ from multivariate Gaussian with covariance depending on the history of the chain of length $k$. In our case, $\theta$ has dimension 59,317, so this approach is infeasible, but we find this simplified version works well. We use $s=2.4$ in \eqref{eq:AdaptiveQ}, the default recommended in \cite{haario2001adaptive}. 

Figure \ref{fig:ESMaxPoint} shows effective samples per second $T_e/t$ for P\'{o}lya-Gamma data augmentation and the algorithm in \eqref{eq:MetropolisUpdate}. The histogram shows the distribution of $T_e/t$ for all 59,317 $\theta_i$ parameters. This shows definitively that the algorithm in \eqref{eq:MetropolisUpdate} has lower computational complexity by orders of magnitude for this application. 

\begin{figure}[h]
\centering
\includegraphics[width=0.7\textwidth]{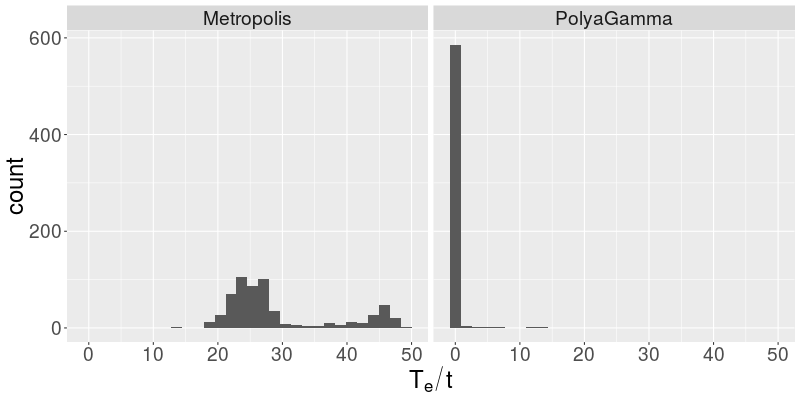} 
\caption{$T_e/t$ for  data augmentation and the algorithm in \eqref{eq:MetropolisUpdate}} \label{fig:ESMaxPoint}
\end{figure}

\section{Discussion}
For several decades, there has been substantial interest in easy to implement and reliable algorithms for posterior computation in generalized linear models. Data augmentation Gibbs sampling, particularly for probit and logit links, has received much of this attention. A series of data augmentation schemes \cite{fruhwirth2010data, holmes2006bayesian, o2004bayesian}, with \citet{polson2013bayesian} being the most popular of the recently developed algorithms, have steadily improved the accessibility of Gibbs sampling for logistic regression. This is a specific case of the larger focus of Bayesian computation on Gibbs samplers, in algorithm development, routine use, and theoretical analysis. The appeal of Gibbs samplers is largely due to their conceptual simplicity, minimal tuning, and widespread familiarity. In addition, there is a common misconception among practitioners that Gibbs samplers are more efficient than alternative Metropolis-Hastings algorithms.

The literature studying theoretical efficiency of Gibbs samplers has largely focused on showing uniform or geometric ergodicity. These bounds often say nothing about the computational complexity of the algorithm.  Here, we obtain upper and lower bounds on computational complexity that explain empirical performance. Although we compare DA and Metropolis for a specific model and imbalanced data setting, the insight we obtain -- that concentration of step sizes may occur at a different rate than concentration of the target -- is relevant generally. We showed step sizes for DA concentrating at rate $1/\sqrt{n}$ but the target concentrating at the slower rate $1/\log n$. Had we selected the ``standard'' asymptotic framework, the target would naturally have concentrated at the usual $1/\sqrt{n}$ rate, implying that DA and Metropolis would have similar performance. This result is uninformative for understanding the failure of DA in large imbalanced data settings, highlighting the importance of moving beyond standard asymptotics when studying algorithmic complexity of MCMC.

\begin{appendix}

\section*{Appendix}
In the proofs, we use $\bigO{g(n)}$ and $\Omega(g(n))$ notation. A function $f(n) = \bigO{g(n)}$ indicates that there exist $C, n_0 < \infty$ such that $n > n_0$ implies $f(n) < C g(n)$. Conversely, $f(n) = \Omega(g(n))$ means that there exists $n_0, C$ such that $n > n_0$ implies $f(n) > C g(n)$. 

\section{Proof of Theorem \ref{thm:mh}}
\subsection{Introduction} \label{SecProbDesc}

In this section, we prove that a properly-tuned Metropolis-Hastings (M-H) algorithm converges quickly when targeting the distribution proportional to
\be 
p_{n}(\theta) = p(\theta) \equiv (1 + e^{\theta})^{-n} \, e^{\theta} \, e^{-\frac{\theta^{2}}{2B}},
\ee 
where $0 < B < \infty$ is a constant and $n \in \mathbb{N}$ is a very large integer.

Our bound will be given for the following Markov chain:

\begin{definition} [Metropolis-Hastings Kernel]
For a fixed sequence $\{ \epsilon_{n} \}_{ n \in \mathbb{N}}$ of strictly positive real numbers, we define the kernel $\P_{n}$ to be the Metropolis-Hastings kernel on $\mathbb{R}$ with proposal kernel
\be 
L_{n}(x,\cdot) = \mathrm{Unif}([x - \epsilon_{n}, x+ \epsilon_{n} ])
\ee 
and target distribution $p_{n}$. 

Throughout this section, we denote by
\be 
\epsilon_{n} &= \log(n) \\
\alpha_{n}(\theta_{1},\theta_{2}) &= \min \left( 1, \frac{p(\theta_{2})}{p(\theta_{1})} \right)
\ee 
the \textit{step size} and \textit{acceptance probability} of the kernel $\P_{n}$, where $p(\cdot)$ is the density of the target $\Pi$. In the interest of using natural notation and avoiding decorations on frequently used symbols, throughout this proof we let $x$ and $y$ represent generic points in $\bb R$ and $X$ represent random variables; we hope this does not cause any confusion with notation used for data in the main text. 
\end{definition}

\subsection{Preliminary Calculations} 

We define $\widehat{\theta} = \mathrm{argmax}_{\theta} p(\theta)$. By straightforward calculus, $\widehat{\theta}$ satisfies 
\be 
\frac{\widehat{\theta}}{B} + n \frac{e^{\widehat{\theta}}}{1 + e^{\widehat{\theta}}} = 1,
\ee 
and so 
\be \label{IneqThetaMax}
\widehat{\theta} = -\log(n) + \bigO{\log(\log n )}.
\ee

We note that $p(\theta)$ has only one local maximum:

\begin{lemma} [Unimodality of $p(\theta)$] \label{LemmaUnimodality}

We have 
\be 
p'(\theta) &>0, \qquad \theta < \hat{\theta} \\
p'(\theta) &<0, \qquad \theta > \hat{\theta}. \\
\ee 

\end{lemma}
\begin{proof}
By direct calculation,
\be 
p'(\theta) &= \left(1 - n \, \frac{e^{\theta}}{1 + e^{\theta}} - \frac{\theta}{B} \right) \, p(\theta). 
\ee 
Define 
\be \label{EqDefF}
f(\theta) = \left(1 - n \, \frac{e^{\theta}}{1 + e^{\theta}} - \frac{\theta}{B} \right),
\ee 
so that 
\be \label{EqDefpPrime}
p'(\theta) = f(\theta) \, p(\theta).
\ee  
We then have that 
\be 
f'(\theta) = -\frac{1}{B} - n \, \frac{e^{\theta}}{1 + e^{\theta}} \left( 1 - \frac{e^{\theta}}{1 + e^{\theta}} \right) \leq - \frac{1}{B} < 0.
\ee 
Since $f'(\theta) < 0$ for all $\theta \in \mathbb{R}$, it follows that $\{ \theta \, : \, f(\theta) = 0\}$ has at most one point. Since $p(\theta) > 0$ for all $\theta \in \mathbb{R}$, we have by Equation \eqref{EqDefpPrime} that 
\be 
\{ \theta \, : \, p'(\theta) = 0\} =  \{ \theta \, : \, f(\theta) = 0\}.
\ee 
Thus, $\{ \theta \, : \, p'(\theta) = 0\}$ has at most one point. Since $p'(\hat{\theta}) = 0$, the lemma follows immediately.
\end{proof}

We bound the \textit{acceptance probability} far from $\hat{\theta}_{n}$:

\begin{lemma} [Bound on Acceptance Probability]
For $\hat{\theta} \leq x \leq y$,
\be \label{IneqUpperBoundAlpha1}
\alpha_{n}(x,y) \leq  e^{-\frac{(y-x)(x+y- 2 \hat{\theta})}{2B}} \leq e^{-\frac{x - \hat{\theta}}{B} (y-x)}.
\ee

For $\hat{\theta} \geq x \geq y$,
\be \label{IneqUpperBoundAlpha2}
\alpha_{n}(x,y) \leq e^{-\frac{(y-x)(x+y- 2 \hat{\theta})}{2B}} \leq e^{-\frac{x - \hat{\theta}}{B} (y-x)}.
\ee 
\end{lemma}

\begin{proof}
We prove Inequality \eqref{IneqUpperBoundAlpha1} first. Define $f$ as in Equation \eqref{EqDefF}. Recall from Equation \eqref{EqDefpPrime} that $p'(\theta) = f(\theta) \, p(\theta),$ and by direct calculation
\be 
f'(\theta) &= -\frac{1}{B} - n \, \frac{e^{\theta}}{1 + e^{\theta}} \left( 1 - \frac{e^{\theta}}{1 + e^{\theta}} \right) \leq - \frac{1}{B} < 0 \\
f(\hat{\theta}) &= 0.
\ee 
Thus, for all $z \geq 0$, $f(z + \hat{\theta}) \leq - \frac{z}{B}$. Combining this with Equality \eqref{EqDefpPrime}, we have 
\be 
p'(z + \hat{\theta}) \leq - \frac{z}{B} p(z + \hat{\theta})
\ee 
for all $z \geq 0$. Let $q \, : \, \mathbb{R}^{+} \mapsto \mathbb{R}^{+}$ be the solution to the ODE
\be 
q(0) &= p(x), \quad q'(z) = -\frac{z+(x-\hat{\theta})}{B} q(z).
\ee 
Note that $q(0) = p(x)$ and $q'(z) \leq p'(z + (x-\hat{\theta}))$ for all $z \geq 0$. Thus, by Gronwall's inequality,
\be \label{IneqGronwallConclusion}
p(y) \leq q(y-x).
\ee 
Solving the ODE that defines $q$, we have 
\be 
q(z) = C e^{-\frac{z(2(x-\hat{\theta}) + z)}{2B}}
\ee 
for some $C \in \mathbb{R}$. Solving for $C$, we have 
\be 
p(x) = q(0) = C,\quad \text{ so } \quad q(z) = p(x)  e^{-\frac{z(2(x-\hat{\theta}) + z)}{2B}}.
\ee
By Inequality \eqref{IneqGronwallConclusion}, this implies 
\be 
p(y) &\leq q(y-x) = p(x) e^{-\frac{(y-x)(2 (x - \hat{\theta}) + (y-x))}{2B}} \\
&= p(x) e^{-\frac{(y-x)(x+y- 2 \hat{\theta})}{2B}}.
\ee 
This completes the proof of Inequality \eqref{IneqUpperBoundAlpha1}. The proof of  Inequality \eqref{IneqUpperBoundAlpha2} is essentially identical.
\end{proof}

\subsection{Drift Bounds}

We show that a Markov chain evolving according to $\P_{n}$ will tend to drift towards $\hat{\theta}$:

\begin{lemma} \label{LemmaLyapunovFunction}
Define $c_{n} \equiv 1$, $\delta_{n} = \sqrt{\log(n)}$, and $V_{n} \, : \, \mathbb{R} \mapsto \mathbb{R}^{+}$ by 
\be 
V_{n}(\theta) = e^{c_{n} \, | \theta  - \hat{\theta}_{n}|}.
\ee 
Let $x \in \mathbb{R}^{+}$ and let $X \sim \P_{n}(x,\cdot)$. Then for all $n > N_{0}$ sufficiently large,
\be \label{IneqLyapunovMain}
\E[V_{n}(X)] \leq \frac{2}{3} \,  V_{n}(x)  + e^{\log(n)}.
\ee

\end{lemma}

\begin{proof}
We prove this in three cases:  $x > \hat{\theta}_{n} + \max(\epsilon_{n},\delta_{n})$, $x < \hat{\theta}_{n} - \max(\epsilon_{n},\delta_{n})$, \,and $\hat{\theta}_{n} - \max(\epsilon_{n},\delta_{n}) \leq x \leq \hat{\theta}_{n} + \max(\epsilon_{n},\delta_{n})$. 

We calculate:
\begin{enumerate}
\item \textbf{Case 1:} $x > \hat{\theta}_{n} + \max(\epsilon_{n},\delta_{n})$. \\

In this case,
\be 
2 \epsilon_{n} \,  \E[V_{n}(X)] &= \int_{x-\epsilon_{n}}^{x} V_{n}(y) \alpha_{n}(x,y) dy + \int_{x}^{x+\epsilon_{n}} V_{n}(y) \alpha_{n}(x,y) dy \\
&+ V_{n}(x) \, \int_{x-\epsilon_{n}}^{x+\epsilon_{n}} ( 1 - \alpha_{n}(x,y)) dy\\
&= \int_{x-\epsilon_{n}}^{x} V_{n}(y) dy + \int_{x}^{x+\epsilon_{n}} V_{n}(y) \alpha_{n}(x,y) dy \\
&+ V_{n}(x) \, \int_{x}^{x+\epsilon_{n}} ( 1 - \alpha_{n}(x,y)) dy \\
&= \frac{1}{c_{n}} (1 - e^{-c_{n} \epsilon_{n}}) \, V_{n}(x) + \int_{x}^{x+\epsilon_{n}} V_{n}(y) \alpha_{n}(x,y) dy \\
&+ V_{n}(x) \, \int_{x}^{x+\epsilon_{n}} ( 1 - \alpha_{n}(x,y)) dy, \\
\ee 
where the inequality in the second line follows from Lemma \ref{LemmaUnimodality}. Using Inequality \eqref{IneqUpperBoundAlpha1}, we continue by writing:
\be 
2 \epsilon_{n} \,  \E[V_{n}(X)] &\leq \frac{1}{c_{n}} (1 - e^{-c_{n} \epsilon_{n}}) \, V_{n}(x) + \int_{x}^{x+\epsilon_{n}} V_{n}(y) \alpha_{n}(x,y) dy \\
&+ V_{n}(x) \, \int_{x}^{x+\epsilon_{n}} ( 1 - \alpha_{n}(x,y)) dy, \\ 
&\leq \frac{1}{c_{n}} (1 - e^{-c_{n} \epsilon_{n}}) \, V_{n}(x) + \int_{x}^{x+\epsilon_{n}} e^{-\frac{\delta_{n}}{B} (y-x)} \, V_{n}(y)  dy \\
&+ V_{n}(x) \, \int_{x}^{x+\epsilon_{n}} ( 1 - e^{-\frac{\delta_{n}}{B} (y-x)}) dy \\
&= \frac{1}{c_{n}} (1 - e^{-c_{n} \epsilon_{n}}) \, V_{n}(x) + \int_{x}^{x+\epsilon_{n}} e^{-\frac{\delta_{n}}{B} (y-x)} \,e^{c_{n}(y-\hat{\theta})}  dy \\
&+ V_{n}(x) \, \int_{x}^{x+\epsilon_{n}} ( 1 - e^{-\frac{\delta_{n}}{B} (y-x)}) dy \\
&= \frac{1}{c_{n}} (1 - e^{-c_{n} \epsilon_{n}}) \, V_{n}(x) + e^{\frac{\delta_{n}}{B}x} e^{-c_{n} \hat{\theta}} \int_{x}^{x+\epsilon_{n}} e^{(c_{n}-\frac{\delta_{n}}{B}) y}  dy \\
&+ V_{n}(x) (\epsilon_{n} - \frac{1}{c_{n}} (1 - e^{-c_{n} \epsilon_{n}})) \\
&= \epsilon_{n} V_{n}(x) + e^{\frac{\delta_{n}}{B}x} e^{-c_{n} \hat{\theta}} \frac{1}{c_{n} - \frac{\delta_{n}}{B}} e^{(c_{n} - \frac{\delta_{n}}{B})x} (e^{(c_{n} - \frac{\delta_{n}}{B}) \epsilon_{n}} - 1) \\
&=  V_{n}(x) (\epsilon_{n} + \frac{1}{c_{n} - \frac{\delta_{n}}{B}} (e^{(c_{n} - \frac{\delta_{n}}{B})\epsilon_{n}}-1)) \\
&= V_{n}(x) \, \epsilon_{n} \, \bigg(1 + \mathcal{O}\bigg(\frac{1}{\sqrt{\log(n)}}\bigg)\bigg).
\ee 
This completes the proof of Inequality \eqref{IneqLyapunovMain} in the first case.

\item \textbf{Case 2:} $x < \hat{\theta}_{n} - \max(\epsilon_{n},\delta_{n})$. \\

The proof of this case is essentially identical to the proof of the first case, with the exception that Inequality \eqref{IneqUpperBoundAlpha2} is used in place of Inequality \eqref{IneqUpperBoundAlpha1}, and this change is propogated through the remaining calculations. The details are omitted.

\item \textbf{Case 3:}  $\hat{\theta}_{n} - \max(\epsilon_{n},\delta_{n}) \leq x \leq \hat{\theta}_{n} + \max(\epsilon_{n},\delta_{n})$.

In this case, we have 
\be 
V_{n}(X) &\leq e^{c_{n} \max(\epsilon_{n}, \delta_{n})} \\
&= e^{\log(n)}. 
\ee 

\end{enumerate}

\end{proof}

\subsection{Minorization Condition}

Define the \textit{small set} 
\be 
\mathcal{C}_{n} = \{ \theta \in \mathbb{R} \, : \, V_{n}(\theta) \leq 6 e^{\log(n)} \} = \{ \theta \in \mathbb{R} \, : \,| \theta - \hat{\theta}| \leq \log(n) + \log(6) \}.
\ee

We have:
\begin{lemma} [Minorization on Small Set] \label{LemmaMinSmall}
There exists some constants $c, N_{0} > 0$ and some sequence of probability measures $\{ \mu_{n} \}_{n \in \mathbb{N}}$ so that
\be 
\P_{n}^{2}(x,\cdot) \geq \frac{c}{\log(n)^{2}} \, \mu_{n}(\cdot)
\ee 
for all $n > N_{0}$ and all $x \in \mathcal{C}_{n}$.
\end{lemma} 

\begin{proof}
Let $X_{1} \sim \P_{n}(x,\cdot)$ and then, conditional on $X_{1}$, let $X_{2} \sim \P_{n}(X_{1},\cdot)$.

Define the event $\mathcal{E}_{n} = \{ |X_{1} - \hat{\theta}| < \frac{1}{10} \log(n) \}$. It is clear that there exists some uniform constant $c_{1} > 0$ so that 
\be \label{IneqMin1}
\bb P[\mathcal{E}_{n}] \geq c_{1}.
\ee 
By \eqref{IneqNearConstBd}, there exists some constant $c_{2}$ so that 
\be \label{IneqMin2}
\P_{n}(y,\cdot) \geq \frac{c_{2}}{\log(n)^{2}} \mathrm{Unif}\bigg(\bigg[\hat{\theta}_{n} - \frac{c_{2}}{\log(n)},\hat{\theta}_{n} + \frac{c_{2}}{\log(n)}\bigg]\bigg)
\ee 
for all $|y - \hat{\theta}| \leq \frac{1}{10} \log(n)$. The result now follows, with $\mu_{n} = \mathrm{Unif}([\hat{\theta}_{n} - \frac{c_{2}}{\log(n)},\hat{\theta}_{n} + \frac{c_{2}}{\log(n)}])$, by combining Inequalities \eqref{IneqMin1} and \eqref{IneqMin2}.
\end{proof}

By Theorem 5 of \cite{rosenthal1995minorization}, Lemmas \ref{LemmaLyapunovFunction} and \ref{LemmaMinSmall} together imply that 
\be 
\| \P_{n}^{T}(\theta,\cdot) - \Pi(\cdot) \|_{\mathrm{TV}} \leq \bigg(1 - \frac{c_{1}}{\log(n)^{2}}\bigg)^{\frac{T}{c_{2} \log(n)}} + c_{3}^{T}(e^{|\theta - \hat{\theta}|} + c_{4} n),
\ee 
for some constants $c_{1},c_{2},c_{3},c_{4} > 0$ that do not depend on $n$, where $0 <c_{3} < 1$ and $c_1, c_2$ are distinct from the constants $c_1, c_2$ in the proof of Lemma \ref{LemmaMinSmall}. By Theorem 2 of \cite{roberts1997geometric}, this completes our proof.

\section{Proof of Theorem \ref{thm:mh}}
\subsection{Preparatory results}

The following Corollary to Theorem \ref{ThmLawSok} will be used to obtain upper bounds on the conductance and spectral gap. Observe that because the set that we consider in this Corollary is a product of the entire sample space for $\omega$ with a subset of the sample space for $\theta$, verifying this Corollary immediately gives a bound on the conductance for the $\theta$-marginal chain, which is reversible.
\begin{corollary} \label{CorCondIneq}
Let $(\theta_{t}, \omega_t)$ be a data augmentation Markov chain on state space $\Omega_{1} \times \Omega_{2} \subset \mathbb{R} \times \mathbb{R}^{n}$. Denote by $\mc P = \mc P_{1} \mc P_{2}$ the transition kernel of this chain, where $\mc P_{1}[(\theta, \omega), \Omega_{1} \times \{ \omega \}] = \mc P_{2}[(\theta,\omega), \{\theta \} \times \Omega_{2}] = 1$ for all $(\theta, \omega) \in \Omega_{1} \times \Omega_{2}$. Denote by $\Pi$ the stationary measure of $\mc P$, and denote by $\Pi_{1}$ and $\Pi_{2}$ the marginals of this stationary measure on $\Omega_{1}$ and $\Omega_{2}$; denote by $\mu$, $\mu_{1}$ and $\mu_{2}$ their densities.  Assume that there exists an interval $I = (a,b) \subset \Omega_{1}$ that satisfies
\be 
\Pi_{1}(I) &\geq 1 - \epsilon \label{IneqCor1} \\
c^* \leq \inf_{\theta \in I} \mu_{1}(\theta) &\leq \sup_{\theta \in I} \mu_{1}(\theta) \leq C^* \label{IneqCor2} \\
\sup_{\theta \in I, z \in \Omega_{2}} & \bb P\left[ (\theta_{s+1}-\theta_{s})^{2} > r \zeta \mid (\theta_{s}, \omega_{s}) = (\theta,z) \right]  \leq  r^{-2} + \gamma \label{IneqCor3}
\ee 
for some $\epsilon, \zeta > 0$, $0 \leq \gamma < \infty$ and $0 < c^* < C^* < \infty$, and for all $0 \le r < (1-\epsilon)/(4c^*)$. Since \eqref{IneqCor3} is trivial for $r < 1$, this is equivalent to \eqref{IneqCor3} holding for $1 \le r \leq (1-\epsilon)/(4c^*)$. Assume that $\zeta \leq \frac{1-\epsilon}{4C^*}$. Then 
\begin{align*} 
\delta(\mc P) \leq \kappa(\mc P) \le  \frac{16 C^* \zeta }{(1-\epsilon)^{2}} + \frac{2C^* \gamma}{c^*(1-\epsilon)}.
\end{align*} 
\end{corollary}

\begin{proof}
Let $m = \inf \left\{ x > a \, : \, \int_{a}^{x} \mu_{1}(y) dy \geq \frac{\pi_{1}(I)}{2} \right\} \geq a + \frac{1-\epsilon}{2C^*}$ be the median of the restriction of $\Pi_{1}$ to $I$ and let $S = (a,m] \times \Omega_{2}$. By inequality \eqref{IneqCor2}, 
\be 
\frac{1-\epsilon}{2c^*} \geq m-a \geq \frac{1-\epsilon}{2C^*}.
\ee 
We now bound the conductance $\kappa$ by showing an upper bound on $\kappa(S)$
\be
\kappa(S) &= \frac{\int_{(x,y) \in S} \mc P((x,y),S^{c}) \mu(x,y) dx dy }{\Pi(S) (1 - \Pi(S))} \\
&\leq \frac{4}{(1-\epsilon)^{2}} \int_{(x,y) \in S} \mc P((x,y),S^{c}) \mu(x,y) dx dy \\
&\leq \frac{4}{(1-\epsilon)^{2}} \int_{a}^{m} C^* \left[ \min\left( 1,\frac{\zeta^{2}}{\min(x-a, m-x)^{2}} \right) + \gamma \right] dx
\ee
where in the last step we applied \eqref{IneqCor3} with $r \le \max(x-a,m-x) \le (1-\epsilon)/(4c^*)$ on $[a,m]$. Continuing 
\be
\kappa(S) &\le \frac{8 C^*}{(1-\epsilon)^{2}} \left( \int_{0}^{\zeta}(1 + \gamma)dx + \int_{\zeta}^{\frac{m-a}{2}} \left( \frac{\zeta^{2}}{x^{2}} + \gamma \right) dx \right) \\
&= \frac{8 C^*}{(1-\epsilon)^{2}} \left( \zeta + \zeta^{2}\left(\zeta^{-1} - \frac{2}{m-a} \right) + \gamma \frac{m-a}{2} \right) \\
&\leq \frac{16 C^* \zeta }{(1-\epsilon)^{2}} + \frac{8C^* \gamma }{(1-\epsilon)^{2}} \frac{1 -\epsilon}{4c^*} \\
&= \frac{16 C^* \zeta }{(1-\epsilon)^{2}} + \frac{2C^* \gamma}{c^*(1-\epsilon)}.
\ee
The result now follows immediately from an application of Theorem \ref{ThmLawSok}.
\end{proof}

\subsection{Verifying Corollary \ref{CorCondIneq}} \label{sec:verifycor}
We briefly outline the strategy for showing the three conditions in Corollary \ref{CorCondIneq}. To show \eqref{IneqCor1}, the existence of an interval $I(n)$ satisfying $\pi_1(I(n)) \geq 1-\epsilon$, we first find an interval $I(n)$ containing the posterior mode for large enough $n$ on which the posterior density ratio is bounded below by a constant. Then, we find a second interval $I'(n)$ outside of which the posterior integrates to $o(1)$ and that satisfies $I(n) \subset I'(n)$. By lower bounding the width of $I(n)$ and upper bounding the width of $I'(n)$, we obtain a lower bound on $\pi_1(I(n))$ and bounds on $c^*(n)$ and $C^*(n)$, sequences of constants corresponding to $C^*$ and $c^*$ in \eqref{IneqCor2}. To show \eqref{IneqCor3}, we study the dynamics of the chain on $I(n)$ and use concentration inequalities.

\subsection{Proof of \eqref{IneqPgcondMain} in Theorem \ref{thm:pgcond}}
We prove \eqref{IneqPgcondMain} with an application of  Corollary \ref{CorCondIneq}. \eqref{IneqAccondMain} is proved in the Supplement. The proof consists of verifying the three conditions given by inequalities \eqref{IneqCor1}, \eqref{IneqCor2}, and \eqref{IneqCor3}. The proof proceeds in three parts:
\begin{enumerate}[(a)]
\item Showing an interval $I(n)$ on which the posterior density ratio is bounded by a constant and lower bounding its width;
\item Showing an interval $I'(n) \supset I(n)$ outside of which the posterior integrates to $o(1)$ and upper bounding its width; and,
\item Showing a concentration result of the form \eqref{IneqCor3} on an interval containing $I(n)$. 
\end{enumerate}

\noindent \textbf{Part (a) : showing the posterior is almost constant on an interval $I(n)$ containing the mode and lower bounding its width.} First, we provide bounds of the form \eqref{IneqCor1} and \eqref{IneqCor2}. Recall that the posterior density of $\theta$ is 
\be 
p(\theta | y = 1) = \frac{n}{(2 \pi)^{1/2} B} (1 + e^{\theta})^{-n} e^{\theta} e^{-\frac{\theta^{2}}{2 B}}.
\ee 

We begin by showing that $p(\theta | y = 1)$ is near-constant on a small region around the mode $\widehat{\theta} \equiv \mathrm{argmax}_{\theta} p(\theta | y = 1$) of width $\Omega( (\log n)^{-1})$ given by $I(n) = [\widehat{\theta}-(\log(n))^{-1},\widehat{\theta}+(\log(n))^{-1}]$. Recall from Inequality \eqref{IneqThetaMax} 
\be 
\widehat{\theta} = \hat{\theta}_{n} = - \log(n) + \mathcal{O}(\log(\log(n))).
\ee 

Therefore, there exists an $A< \infty$ such that $\widehat{\theta} \in [-\log(n) - A\log(\log n),-\log(n) + A\log(\log n)]$ for all $n>N_0$, where $N_0$ depends only on $A$.

Consider pairs $\theta_{1}, \theta_{2}$ that satisfy $|\theta_{1} - \theta_{2}|  \leq (\log n)^{-1}$ and also $|\theta_{1} + \log(n)|, \, |\theta_{2} + \log(n) | \leq A \log(\log n )$. Define $\zeta_{1}, \zeta_{2}$ by $\theta_{1} = - \log(n) + \zeta_{1}$, $\theta_{2} = - \log(n) + \zeta_{2}$. Then we calculate 
\be \label{IneqNearConstBd}
\frac{p(\theta_{1} |  y = 1)}{p(\theta_{2} |  y = 1)} &= e^{\theta_{1}-\theta_{2}} \left( \frac{1 + e^{\theta_{2}}}{1 + e^{\theta_{1}}} \right)^{n} e^{\frac{1}{2B}(\theta_{2}^{2} - \theta_{1}^{2})} \\
&= e^{\zeta_{1}-\zeta_{2}} \left( \frac{1 + \frac{1}{n} e^{\zeta_{2}}}{1 + \frac{1}{n} e^{\zeta_{1}}} \right)^{n} e^{\frac{1}{2B}(\zeta_{1}-\zeta_{2})(2 \log(n) - \zeta_{1} - \zeta_{2})} \\
&\geq (e^{-2})(2e)^{-2A}(e^{-2/B}).
\ee 
Since this holds for any pair of points satisfying $|\theta_1-\theta_2| \leq (\log n)^{-1}$ inside the interval $-\log(n) \pm A\log(\log n)$, and $\widehat{\theta}$ is inside this interval for $n > N_0$, we conclude the posterior density ratio is bounded below by a constant on an interval $I(n)$ of width $\Omega((\log n)^{-1})$ centered at $\widehat{\theta}$ for all $n > N_0$. Since the posterior density must integrate to 1, this shows that $\mu_1(\widehat{\theta}) = \bigO{\log n}$ in \eqref{IneqCor2}, so we can take $C^*(n)=\bigO{\log n}$ in \eqref{IneqCor2}.

\textbf{Part (b): showing the posterior is negligible outside an interval $I'(n) \supset I(n)$ and upper bounding its width}. Next, we show that $p(\theta | y = 1)$ is negligible outside of the interval $I'(n)=(-5 \log(n), 3 \log(n))$. This interval clearly contains $I(n)$ for all large $n$, since $I(n)$ is an interval of width $\bigO{\log(\log n)}$ containing $-\log(n)$. If $\theta = -\log(n) + C \log(n)$ for some $C \geq 4$,
\be 
p(\theta | y = 1) &\leq \frac{n}{(2 \pi)^{1/2} B} n^{C-1} (1 + n^{C-1})^{-n} e^{- \frac{(C-1)^{2} (\log(n))^{2}}{2B}} \\
&\leq \frac{1}{(2 \pi)^{1/2}B} n^{C - n(C-1) - \frac{(C-1)^{2}}{2B} \log(n)}.
\ee 
Thus,
\be \label{IneqNegligibleCBig}
\int_{3 \log(n)}^{\infty} p(\theta | y = 1) d \theta &\leq \sum_{C=4}^{\infty}  \frac{\log(n)}{(2 \pi)^{1/2}B} n^{C - n(C-1) - \frac{(C-1)^{2}}{2B} \log(n)} = o(1).
\ee 
If $\theta = -\log(n) - C \log(n)$ for some $C \geq 4$, then
\be 
p(\theta | y = 1) &\leq \frac{n}{(2 \pi)^{1/2} B} n^{-C-1} (1 + n^{-C-1})^{-n} e^{- \frac{(C+1)^{2} (\log(n))^{2}}{2B}} \\
&\leq \frac{2}{(2 \pi)^{1/2}B} n^{-C  - \frac{(C+1)^{2}}{2B} \log(n)}.
\ee 
Thus,
\be \label{IneqNegligibleCSmall}
\int_{- \infty}^{-5 \log(n)} p(\theta | y = 1) d \theta &\leq \sum_{C=4}^{\infty} \frac{2 \log(n)}{(2 \pi)^{1/2}B} n^{-C  - \frac{(C+1)^{2}}{2B} \log(n)} = o(1).
\ee 
Combining inequalities \eqref{IneqNegligibleCBig} and \eqref{IneqNegligibleCSmall} gives
\be \label{IneqNegligibleGeneral}
\int_{\{I'(n)\}^{c}} p(\theta | y = 1) d \theta = o(1).
\ee 
Therefore, since the posterior is negligible outside a region of width $\bigO{\log n}$, and the density is unimodal and smooth, we can take $c=\Omega((\log n)^{-1})$ in \eqref{IneqCor2}. This also shows that $\pi_1(I(n)) = \Omega((\log n)^{-2})$, so we can take $1-\epsilon(n) = \Omega((\log n)^{-2})$ in \eqref{IneqCor1}.

\noindent \textbf{Part (c): Showing \eqref{IneqCor3} on an interval containing the mode.} Fix a constant $0<C<1$ and consider the interval 
\be
I^*(n) = [-\log(n)(1+C),-\log(n)(1-C)]. \label{eq:concintervallogit}
\ee
This interval contains $\widehat{\theta} \in -\log(n) \pm \bigO{\log( \log n)}$ for sufficiently large $n$. We will show \eqref{IneqCor3} on $I^*(n)$. We can write values of $\theta_t$ inside this interval as $\theta_{t} = -\log(n)(1 + a_{t})$ for $|a_{t}| \leq C$. Recall that we are considering the P\'{o}lya-Gamma sampler with an update rule consisting of sampling $\omega_{t+1} \mid \theta_t, n$ and then sampling $\theta_{t+1} \mid \omega_{t+1}, y, n$. We first obtain bounds on the conditional expectation and variance of $\omega_{t+1} \mid \theta_t, n$ for $\theta_t$ inside of $I^*(n)$, which will be used to show concentration. We have

\be 
E( \omega_{t+1} \mid \theta_t, n) &= \frac{n}{2 \theta_{t}} \mathrm{tanh}\left(\theta_{t}/2\right) \\
&= \frac{n}{ -2 \log(n)(1 + a_{t})} \frac{1 - e^{\log(n) (1 + a_{t})}}{1 +  e^{\log(n) (1 + a_{t})}} \\
&= \frac{n}{ -2 \log(n)(1 + a_{t})}  \frac{1 - n^{1 + a_{t}}}{1+n^{1+ a_{t}}} \\
&= \frac{n}{ 2 \log(n)(1 + a_{t})}\left[ 1 - 2n^{-1-a_{t}}(1- o(1)) \right] \label{CalcExpWNext}
\ee 
and 
\be 
\mathrm{var}(\omega_{t+1} \mid \theta_t, n) &= \frac{n}{4 \theta_{t}^{3}} (\mathrm{sinh}(\theta_{t}) - \theta_{t}) \mathrm{sech}^{2}\left(\frac{\theta_{t}}{2} \right) \\
&= \frac{-n}{4 (1 + a_{t})^{3} \log(n)^{3}} \left[ \frac{1 - e^{2 (1 + a_{t}) \log(n)}}{2 e^{(1 + a_{t}) \log(n)}} + (1 + a_{t}) \log(n)\right] \\
&\times \left[ \frac{2 e^{\frac{1}{2} (1 + a_{t}) \log(n)}}{ 1 + e^{(1+ a_{t}) \log(n)}}  \right]^{2} \\
&= \frac{n}{4 (1 + a_{t})^{3} \log(n)^{3}} \left[ \frac{1}{2} n^{1 + a_{t}}(1 + o(1)) \right] \left[ \frac{4}{n^{1 + a_{t}}}( 1 + o(1)) \right] \\
&= \frac{n}{2 (1 + a_{t})^{3} \log(n)^{3}}( 1 + o(1)). \label{CalcVarWNext}
\ee 

Define $\zeta(n) = n^{1/2} (\log n)^{-1.5}$. Combining \eqref{CalcExpWNext} and \eqref{CalcVarWNext}, we have by Chebyshev's inequality that
\be \label{EqOneStepW}
\pr\left( \left|\omega_{t+1} - \frac{n}{2 \log(n) (1 + a_{t})} \right| > r \frac{n^{1/2}}{\log(n)^{1.5}} \,\bigg|\, \theta_t  \right) = \bigO{r^{-2}}
\ee 
for any $r > 0$, $\theta_t \in I^*(n)$. Next, we bound $|\theta_{t+1}-\theta_t|$ for $\theta_t \in I^*(n)$. Recall  
\be 
\theta_{t+1}|\omega_{t+1} \sim \No{ \sigma_{\omega_{t+1}}^{-1} (y-n/2)}{\sigma_{\omega_{t+1}}^{-1} }, \quad  \sigma_{\omega_{t+1}}^{-1} = (\omega_{t+1} +  B^{-1})^{-1}.
\ee

Define $r_{t}$ by $\omega_{t+1} = n(2 \log(n) (1 + a_{t}))^{-1} + r_{t} n^{1/2} \log(n)^{-1.5}$. In the following, we condition on $r_{t} (4 \log(n))^{1/2}n^{-1/2} \leq 1/8$ and $4 B^{-1} \log(n) n^{-1}  \leq 1/8$ to obtain concentration results. Clearly, the second condition holds for fixed $0<B < \infty$ for all sufficiently large $n$. To show that the second condition holds for the relevant values of $r_t$, recall that we need to show \eqref{IneqCor3} only for $1 \le r \le (1-\epsilon)/(4c^*)$. Since $1-\epsilon(n) \le 1$ and $c^*(n) = \Omega((\log n)^{-1})$, $r \le (1-\epsilon(n))/(4c^*(n))$ gives $r = \bigO{\log(n)}$, so $r_{t} (4 \log(n))^{1/2}n^{-1/2} = o(1)$, as required. 

Conditional on $r_{t} (4 \log(n))^{1/2}n^{-1/2} \leq 1/8$ and $4 B^{-1} \log(n) n^{-1}  \leq 1/8$, we have 
\be 
\sigma_{\omega_{t+1}}^{-1} &= \left[ \frac{n}{2 \log(n) (1 + a_{t})} + r_{t} \frac{n^{1/2}}{(\log n)^{1.5}} +  B^{-1} \right]^{-1} \\
&= \frac{2 \log(n) (1 + a_{t})}{n} \left[1 + B^{-1} \frac{2 \log(n) (1 + a_{t})}{n} + r_{t} \frac{2  (1 + a_{t})}{(n \log n)^{1/2}}\right]^{-1} \\
&=\frac{2 \log(n) (1 + a_{t})}{n} \left[ 1 -\bigO{B^{-1} \frac{2 \log(n) (1 + a_{t})}{n} + r_{t} \frac{2  (1 + a_{t})}{(n \log n)^{1/2}} } \right] \\
&=\frac{2 \log(n) (1 + a_{t})}{n}\left[ 1 - \bigO{ \frac{r_{t} + 1}{(n \log n)^{1/2}}} \right].
\ee 

Thus, still conditional on $r_{t} (4 \log n)^{1/2} n^{-1/2} \leq 1/8$ and $4 B^{-1} \log(n) n^{-1}  \leq 1/8$, 
\be 
\theta_{t+1}|\omega_{t+1} \sim &\No{ \sigma_{\omega_{t+1}}^{-1} (y-n/2)}{  \sigma_{\omega_{t+1}}^{-1} } \\
= &\Nor \bigg(  (2-n) \frac{ \log(n) (1 + a_{t})}{n}\left[1 + \bigO{\frac{r_{t} + 1}{(n \log n)^{1/2}}} \right], \frac{ \log(n) (1 + a_{t})}{n}\left[1 + \bigO{\frac{r_{t} + 1}{(n \log n)^{1/2}}}\right] \bigg) \\
= &\Nor \bigg(  - \log(n) (1 + a_{t})\left[ 1 + \bigO{\frac{r_{t} + 1}{(n \log n)^{1/2}}}\right], \frac{ \log(n) (1 + a_{t})}{n}\left[ 1 + \bigO{\frac{r_{t} + 1}{(n \log n)^{1/2}}}\right] \bigg) \\
= &\No{ \theta_{t}\left[1 + \bigO{\frac{r_{t} + 1}{(n \log n)^{1/2}}}\right]}{ \frac{ \log(n) (1 + a_{t})}{n}\left[1 + \bigO{\frac{r_{t} + 1}{(n \log n)^{1/2}}} \right]}.
\ee 

Applying this bound along with Chebyshev's inequality to the second term, and applying inequality \eqref{EqOneStepW} to the first term, we conclude that
\be \label{IneqSimpleConclusion}
\pr \left[ \left| \theta_{t+1} - \theta_{t} \right| > \frac{2r \log n}{n} \right] &\leq \pr \left[ \left|\omega_{t+1} - \frac{n}{2 \log(n) (1 + a_{t})} \right| >  \frac{r n^{1/2}}{(\log n)^{1.5}} \right]  \\
&+ \pr \bigg[ | \theta_{t+1} - \theta_{t} | > 2r \left(\frac{\log n}{n}\right)^{1/2} \quad \bigg| \\
&\left|\omega_{t+1} - \frac{n}{2 \log(n) (1 + a_{t})} \right| \leq r \frac{n^{1/2}}{(\log n)^{1.5}} \bigg] \\
&= \bigO{r^{-2}} + \bigO{\left(\frac{\log n }{n}\right)^{1/2}}.
\ee 

Thus, inequality \eqref{IneqCor3} is satisfied for two sequences of constants $\zeta = \zeta(n)$ and $\gamma = \gamma(n)$ that satisfy
\be \label{IneqEquationsSatisfiedOne}
\zeta(n) = \bigO{\left(\frac{\log n}{n}\right)^{1/2}}, \, \, \gamma(n) =  \bigO{\left(\frac{\log n}{n}\right)^{1/2}}
\ee
on any sequence of sets $I^* = I^*(n)$ satisfying $I^*(n) \subset (-\log(n)(1 + C), -\log(n)(1 - C))$ and fixed $0 < C < 1$.

By inequalities \eqref{IneqNearConstBd} and \eqref{IneqNegligibleGeneral}, the Inequalities \eqref{IneqCor1} and \eqref{IneqCor2} are satisfied with $\epsilon = \epsilon(n)$, $c = c(n)$ and $C = C(n)$ satisfying
\be \label{IneqEquationsSatisfiedTwo}
(1-\epsilon(n))^{-1} = \bigO{(\log n)^{2}}, \, \, c^*(n) = \Omega( (\log n)^{-1}), \, \, C^*(n) = \bigO{\log n}
\ee 
and a set $I(n) \subset \left(-\log(n) - (\log n)^{-1} - \eta_{n}, -\log(n) + (\log n)^{-1} - \eta_{n} \right)$, where by Inequality \eqref{IneqThetaMax} we have $\eta_{n} = \bigO{\log(\log n)}$. Combining this with \eqref{IneqEquationsSatisfiedOne}, Corollary \ref{CorCondIneq}
 completes the proof of Equation \eqref{IneqPgcondMain}.

Finally, Equality \eqref{IneqPgcondWarm} follows immediately from  inequalities \eqref{IneqNearConstBd} and \eqref{IneqNegligibleGeneral}. This completes the proof of the Theorem.
\end{appendix}

\section{Adaptive hybrid Metropolis algorithm}
We give the full algorithm we use for the model in \eqref{eq:HierModel}.
\begin{enumerate}
 \item Update $\theta_i \mid \theta_0, \sigma, y, n$ for $i=1,\ldots,N$ independently in parallel Metropolis steps using the adaptive proposal outlined in \eqref{eq:AdaptiveQ}.
 \item Update $\theta_0 \mid \theta_1,\ldots,\theta_N, \sigma, y, n$ from
 \be
 \Nor( s m, s), \quad s = \left( \frac{N}{\sigma^2} + \frac{1}{B} \right)^{-1}, \quad m = \frac{\sum_{i=1}^N \theta_i}{\sigma^2} + \frac{b}{B}.
 \ee
 \item Update $\sigma \mid \theta_0,\theta_1,\ldots,\theta_N, y,n$ using slice sampling as in \cite[Supplement]{polson2014bayesian}. Specifically, put $\eta = \sigma^{-2}$, then sample 
\be
u \sim \text{Uniform}\left(0,\frac{1}{\eta+1} \right),
\ee
then sample $\eta$ from an exponential distribution with scale
\be
\frac{\sum_{i=1}^N (\theta_i-\theta_0)^2}{2}
\ee
truncated to the interval $\left( 0,\frac{1-u}{u} \right)$. Now put $\sigma^2 = \eta^{-1}$, giving a new sample of $\sigma^2$.
\end{enumerate}

\beginsupplement
\section*{Supplementary Material}

\section{Computational efficiency of MCMC}
This section provides a more detailed introduction to Markov chain concepts relevant to our study of computational complexity and the results in Section \ref{sec:theory}.

Let $\{\Theta_k\}$ be a Markov chain with transition kernel $\mc P$ and target measure $\Pi$. 
For $f : \mc T \to \bb R$, the expectation $\Pi f$  is usually estimated by the time average 
\be \label{eq:ergavg}
\widehat{f}_T = \frac{1}{T} \sum_{k=0}^{T-1} f(\Theta_k).
\ee
For $\Theta_1 \sim \mu$, the mean squared error of $\widehat{f}_T$ is
\be \label{eq:msemcmc}
\Delta(\widehat{f}_T,\Pi f) &\equiv \bb E\Big[ \big( \widehat{f}_T - \Pi f \big)^2 \Big] \\
&= \big( \bb E [ \widehat{f}_T ] - \Pi f \big)^2  + \bb E\bigg[ \bigg(\widehat{f}_T  - \frac{1}{T} \sum_{k=0}^{T-1} \mu \mc P^{k-1} f \bigg)^2\bigg].
\ee
The right side of \eqref{eq:msemcmc} is analogous to a bias-variance decomposition, with
\be \label{eq:biasvar}
\var[ \widehat{f}_T ] &= \bb E \bigg[ \bigg( \frac{1}{T} \sum_{k=0}^{T-1} f(\Theta_k)- \frac{1}{T} \sum_{k=0}^{T-1} \mc \mu P^{k-1} f \bigg)^2 \bigg], \\
\mathrm{bias}(\widehat{f}_T) &= \bb E [ \widehat{f}_T ] - \Pi f.
\ee
Under fairly general conditions, $\Delta(\widehat{f}_T,\Pi f)$ decreases at the rate $T^{-1}$, but the implied multiplicative constant may be enormous and is intimately related to the convergence properties of $\mc P$. Further, this constant often increases with sample size, so that for large samples, huge MCMC path lengths are necessary to achieve acceptable error. 



\subsection{Spectral gap, conductance, and approximation error}
The $L^2(\Pi)$ spectral gap (henceforth ``spectral gap'') of a Markov operator $\mc P$ is defined as
\begin{definition} [Spectral Gap]
Let $\mc P(\theta;\cdot)$ be the transition kernel of a Markov chain with unique stationary distribution $\Pi$. The \textit{spectrum} of $\mc P$ is 
\be 
S = \{ \lambda \in \mathbb{C} \backslash \{0\} \, : \, (\lambda I - \mc P)^{-1} \text{ is not a bounded linear operator on } L^{2}(\Pi) \},
\ee
where $L^2(\Pi)$ is the space of $\Pi$-square integrable functions. The \textit{spectral gap} of $\mc P $ is given by 
\be 
\delta( \mc P) = 1 - \sup \{ |\lambda| \, : \, \lambda \in S, \, \lambda \neq 1 \}
\ee  
when the eigenvalue 1 has multiplicity 1, and $\delta(\mc P) = 0$ otherwise.
\end{definition}
When $\mc P$ is reversible, one can obtain both finite-time and asymptotic bounds on $\Delta(\widehat{f}_T,\Pi f)$ in terms of $\delta(\mc P)$. The following result from \citet[Theorem 3.41]{rudolf2011explicit} gives a finite-time bound\footnote{\cite{rudolf2011explicit} actually gives a slightly sharper version of this bound, where in the first two terms the spectral gap is replaced by $1-\sup\{\lambda : \lambda \in S, \lambda \ne 1 \}$, but for simplicity we give the bound in terms of spectral gap.}
\begin{theorem} \label{thm:ErrorBoundRudolf}
Let $\mc P$ be a reversible Markov kernel. Suppose $f \in L^p(\Pi)$ for $p \in (2,\infty]$ and $\| \frac{d\nu}{d\Pi} \|_{p/(p-2) \vee 2} < \infty$ for $\nu$ a probability measure on $\Theta$ and $\| \cdot \|_q$ the $L^q(\Pi)$ norm. Then
 \be \label{eq:ErrorBoundRudolf}
 \frac{1}{\|f\|_2} \Delta(\widehat{f}_T,\Pi f) \le \frac{2-\delta}{T \delta} - \frac{2 \bar{\delta} (1-\bar{\delta}^T)}{T^2 \delta^2} + \frac{\|f\|_p^2}{\|f\|_2} \frac{64 p}{T^2 (p-2) \delta^2} \left\| \frac{d \nu}{d \Pi} -1 \right\|_{p/(p-2) \vee 2},
 \ee
 where $\bar{\delta} = 1-\delta$.
\end{theorem}
\noindent As a function of $T$, \eqref{eq:ErrorBoundRudolf} is order $T^{-1}$, while as a function of $\delta$ for finite $T$, it is order $\delta^{-2}$. Taking $T \to \infty$ in \eqref{eq:ErrorBoundRudolf}, we obtain the well-known asymptotic bound
\be \label{eq:ErrorBoundAsymp}
\frac{1}{\|f\|_2} \Delta(\widehat{f}_T,\Pi f) \lesssim \frac{2-\delta}{T \delta} 
\ee
Thus, for large $T$, $T \Delta(\wh f_T, \Pi f)$ behaves like $\delta^{-1}$. It is worth noting that the condition on the Radon-Nikodym derivative $\| d\nu/d\Pi \|_{p/(p-2) \vee 2} < \infty$ is similar to the condition $M < \infty$ in Theorem \ref{ThmBiasMcmc}, and the ``warm start'' distributions $\nu$ that we construct for the kernels considered here satisfy $\| d \nu/d\Pi \|_2 < \infty$, so a bound of the form \eqref{eq:ErrorBoundRudolf} holds for $p \ge 4$ for the cases under study.

One can also obtain a general central limit theorem, which we review here because of its role in MCMC performance diagnostics. The following result from \citet[Corollary 4]{jones2004markov} is valid for the Markov kernels considered here, all of which are geometrically ergodic and reversible. A Markov chain evolving according to $\mc P$ is geometrically ergodic if there exist constants $\rho \in (0,1), B < \infty$ and a $\Pi$-almost everywhere finite measurable function $V : \mc T \to [1,\infty)$ such that
\be
\| \mc P^k(\theta_0;\cdot) - \Pi \|_V \le BV(\theta_0) \rho^k,
\ee
where for a probability measure $\mu$, $\|\mu\|_V = \sup_{f \le V} |\mu f|$. 
For a function $f \in L^2(\Pi)$, define the quantity
\be \label{EqLimitingVariance}
\sigma_{f}^{2} = \mathrm{var}[f(\Theta_0)] + 2 \sum_{k=1}^{\infty} \mathrm{cov}[f(\Theta_0),f(\Theta_k)]
\ee 
for $\Theta_0 \sim \Pi$. Then, if $\mc P$ is reversible we have
\be 
\lim_{T \rightarrow \infty} T^{1/2}\left( \widehat{f}_T - \Pi f \right) \stackrel{d}{=} \Nor(0,\sigma_{f}^{2}), \label{eq:markovclt}
\ee 
for any initial distribution $\mu$ on $\Theta_0$. 
The asymptotic variance $\sigma^2_f$ is related to the spectral gap by the inequality (c.f. page 479 of \cite{geyer1992practical}) 
\be
\frac{\sigma_{f}^{2}}{\mathrm{var}_{\Pi}(f)} &\equiv 1 + 2 \sum_{k=1}^{\infty} \eta_{k}(f) \le \frac{2}{\delta(\mc P)}-1, \label{eq:corvar}
\ee
which is identical to \eqref{eq:ErrorBoundAsymp}. The quantity $2 \sum_{k=1}^{\infty} \eta_k(f)$ is referred to as the integrated autocorrelation time.  Here, $\eta_k(f)$ is the lag-$k$ autocorrelation $\cor{f(\Theta_0)}{f(\Theta_k)}$ with $\{\Theta_k\}$ evolving according to $\mc P$. A commonly used performance measure for MCMC is the \emph{effective sample size} $\Teff$:
\be
\Teff &\equiv \frac{\mathrm{var}_{\Pi}(f) T}{\sigma^2_f}, \label{eq:teff} 
\ee 
which measures how much the asymptotic variance is inflated by autocorrelation.
The bound in \eqref{eq:corvar} is sharp for worst case functions when $\mc P$ has no residual spectrum, which holds, for example, for reversible Markov operators on discrete state spaces.  So to a first approximation, the effective sample size is proportional to $\delta(\mc P)$. Since $\Teff$ is an asymptotic quantity, it is common to approximate it from finite-length sample paths using finite-time estimates of the spectral density at frequency zero. Taken together, \eqref{eq:markovclt} and \eqref{eq:ErrorBoundRudolf} indicate that $\Delta(\wh f_T, \Pi f)$ is proportional to $\delta(\mc P)^{-2}$ for finite $T$, and $\delta(\mc P)^{-1}$ asymptotically (as $T \to \infty$). We will generally refer to the asymptotic setting when discussing results.

The bias of finite-length Markov chains can be bounded in terms of the \emph{conductance} of the associated $\mc P$, which also provides a double-sided bound on $\delta$.
\begin{definition} [Conductance] \label{def:conductance}
For $\Pi$-measurable sets $S \subset \Theta$ with $0 < \Pi(S) < 1$, define
\begin{align*}
\kappa(S) = \frac{\int_{\theta \in S} \mc P(\theta,S^{c}) \Pi(ds)}{\Pi(S) (1 - \Pi(S))} 
\end{align*}
and the \emph{Cheeger constant} or \emph{conductance}
\begin{align*}
\kappa = \inf_{0 <\Pi(S) < 1} \kappa(S).
\end{align*}
\end{definition}
The relationship between conductance and bias for Markov operators is quantified by the following Theorem from \citet{lovasz1993random}:
\begin{theorem} [Warm Start Bound] \label{ThmBiasMcmc}
Let $\mc P(\theta;\cdot)$ be the transition kernel of a Markov chain $\{\Theta_j\}_{j \in \mathbb{N}}$ with invariant measure $\Pi$ and conductance $\kappa$. Then for all measurable sets $S \subset \mc T$, 
\begin{align}
| \pr(\Theta_{j+1} \in S) - \Pi(S)| \leq M^{1/2} \left(1 - \frac{\kappa^{2}}{2} \right)^j,  \label{eq:mcmcbias}
\end{align}
where $M = \sup_{A \subset \Theta} \frac{\pr(\Theta_0 \in A)}{\Pi(A)}$. 
\end{theorem}
The total variation bound in \eqref{eq:mcmcbias} is enough to give useful bounds on the bias term in \eqref{eq:biasvar} for bounded $f$. When $\kappa$ is near zero, Theorem \ref{ThmBiasMcmc} implies that the number of steps required to attain bias $\epsilon$ -- the mixing time $\tau_{\epsilon}$ -- scales approximately like $\kappa^{-2}$. For $j \gg \frac{\log(\frac{\epsilon}{\sqrt{M}})}{\log(1 - \frac{\kappa^{2}}{2})} \approx \kappa^{-2} \log\big(\frac{\epsilon}{\sqrt{M}} \big)$, we have 
\be \label{eq:mixtime}
\sup_{ S} | \bb P(\Theta_{t+1} \in S) - \Pi(S) | &\leq \sqrt{M} \bigg(1 - \frac{\kappa^{2}}{2}\bigg)^j \\
&\approx \sqrt{M} \exp\bigg(-\kappa^{2} \times \kappa^{-2} \times \log\bigg(\frac{\epsilon}{\sqrt{M}}\bigg) \bigg) \\
&= \sqrt{M} \times \frac{\epsilon}{\sqrt{M}} = \epsilon.
\ee
For reversible Markov operators $\mc P$, Theorem 2.1 of \cite{lawler1988bounds} relates conductance to the spectral gap:
\begin{theorem} \label{ThmLawSok}
The spectral gap $\delta(\mc P)$ of a reversible Markov operator $\mc P$ satisfies
\be 
\frac{\kappa^{2}}{8} \leq \delta(\mc P) \leq \kappa. \label{eq:delta-kappa}
\ee 
\end{theorem}
Clearly, one gives up a factor of either $\delta$ or $\kappa$ when transitioning between bounds on $\kappa$ and bounds on $\delta$ via \eqref{eq:delta-kappa}. 

\section{Results on mixing times}
We give some additional results on mixing times for the data augmentation algorithms.

The following remark shows that Theorem \ref{ThmBiasMcmc} applies to the P\'{o}lya-Gamma sampler. It is established in the course of proving Theorem \ref{thm:pgcond}.
\begin{corollary}[Warm start for P\'{o}lya-Gamma] \label{cor:wspg}
Let $\hat{\theta}$ be the posterior mode for the model in \eqref{eq:InterceptModel} with $g^{-1}$ the inverse logit link. Then 
\be 
\mu_{n} &= \mathrm{Uniform}\left(\hat{\theta} - \frac{1}{\log(n)}, \hat{\theta} + \frac{1}{\log(n)} \right) \notag \\
\intertext{satisfies}
\sup_{A \subset \mathbb{R}} \frac{\mu_{n}(A)}{\Pi(A|y)} &\lesssim (\log n)^{2}   \label{IneqPgcondWarm}
\ee
where $\Pi$ is the posterior measure, and thus provides a `warm start' distribution for the P\'{o}lya-Gamma sampler. Note that $\hat{\theta}$ is the unique solution to $\frac{\theta}{B} + n \frac{e^\theta}{1 + e^\theta} = 1$. 
\end{corollary}
Combined with \eqref{eq:mixtime}, Corollary \ref{cor:wspg} therefore indicates that $\epsilon$-mixing times for the P\'{o}lya-Gamma sampler scale approximately like $\kappa^2_n(\mc P) \lesssim (\log n)^{11} n^{-1}$. The following corollary shows that Theorem \ref{ThmBiasMcmc} also applies to the Albert and Chib sampler. It is also established in the proof of Theorem \ref{thm:pgcond}.

\begin{corollary}[Warm start for Albert and Chib] \label{cor:wsac}
Let $\Phi$ be the standard Gaussian cumulative distribution function. Then 
$$\mu_n = \mathrm{Uniform}\left(\Phi^{-1}\left(\frac{B+2}{2 (Bn+2)} \right),\Phi^{-1}\left( \frac{2(B+2)}{Bn+2} \right) \right)$$ 
satisfies
\begin{align} \label{IneqAccondWarm}
\sup_{A \subset \mathbb{R}} \frac{\mu_n(A)}{\Pi(A|y)} \lesssim \log n
\end{align} 
and thus provides a `warm start' distribution for the Albert and Chib sampler. 
\end{corollary}
Combined with \eqref{eq:mixtime}, Corollary \ref{cor:wsac} therefore indicates that $\epsilon$-mixing times for the Albert and Chib sampler scale approximately like $\kappa^2_n(\mc P) \lesssim (\log n)^5 n^{-1}$

\section{Additional empirical analysis}
\subsection{Empirical analysis of efficiency}
In this section, we estimate empirically the efficiency of the two data augmentation Gibbs samplers via finite path estimates of $\sigma^2_f$ in \eqref{eq:corvar}.  These empirical estimates can be compared to our estimate based on the conductance that the asymptotic variance is approximately order $n^{1/2}$ to order $n$ for typical functions.  

Let $\sigma^2_f(n)$ be the asymptotic variance when the sample size is $n$, and assume $\sigma^2_f(n) \approx C n^a$ for large $n$, so that $\log\{\sigma^2_f(n)\}=  \log(C) + a \log(n)$. This suggests first estimating $\sigma^2_f(n)$ for different values of $n$, and then estimating $a$ by regression of $\log\{ \sigma^2_f(n)\}$ on $\log(n)$. 

We estimate $\sigma^2_f$ based on autocorrelations via a truncation of (\ref{eq:corvar}), 
\be
\widehat{\rho}_f = 1 + 2 \sum_{k=1}^{K} \widehat{\eta}_k, \label{eq:rhoest}
\ee
where $\widehat{\eta}_k$ is a point estimate of $\eta_k$. It is important to choose $K \gg \{\delta_n(\mc P)\}^{-1}$. The lower bounds derived in Section \ref{sec:theory} have $\{\delta_n(\mc P)\}^{-1} \ge n^{1/2}$ up to a log factor and a universal constant, so we use $K = n$ to compute the sum in (\ref{eq:rhoest}). To further improve the estimates, we use multiple chains to compute $\widehat{\eta}_k$ and run all of the chains for $10^6$ iterations.

Figure \ref{fig:vnk} shows $\log(n)$ versus $\log(\widehat{\rho}_f)$ for values of $n$ between 10 and 10,000 for the P\'{o}lya-Gamma and Albert and Chib sampler. The relationships are linear, and the least squares estimate of the slope is $0.86$ for P\'{o}lya-Gamma and 0.84 for Albert and Chib. This is in the range of $n^{1/2}$ to $n$ estimated from our upper bound on the conductance.
\begin{figure}[h]
	\centering
	\begin{tabular}{cc}
		P\'{o}lya-Gamma &  Albert and Chib \\
		\includegraphics[width=0.4\textwidth]{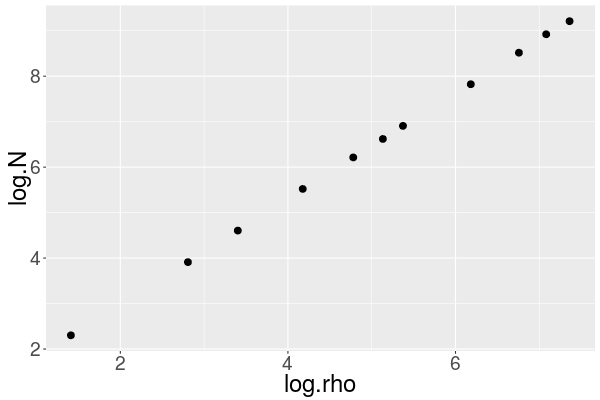} &  \includegraphics[width=0.4\textwidth]{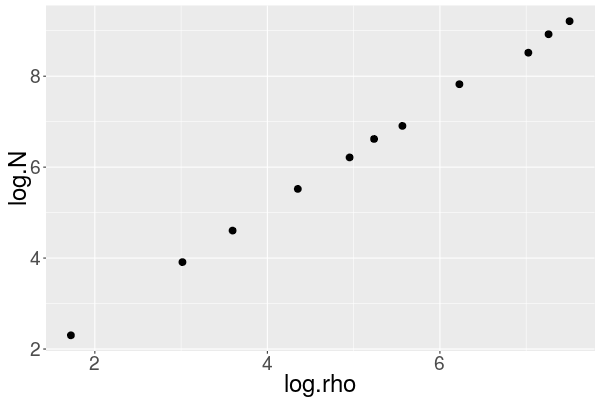} \\
	\end{tabular}
	\caption{Plots of $\log(n)$ versus $\sigma^2(n)$ for different values of $n$. The estimated values of $a$ are $0.86$ and $0.84$, respectively. } \label{fig:vnk}
\end{figure}

\subsection{Data augmentation algorithms for multinomial likelihoods}
So far we have considered data augmentation algorithms for binomial likelihoods. Similar algorithms exist for multinomial logit and probit models. Specifically, let
\be \label{eq:mnlp}
y &\sim \text{Multinomial}(n,\pi), \quad \pi = g^{-1}(\theta),\quad \theta \sim N(0,B),
\ee
where $y$ is a length $d$ vector of nonnegative integers whose sum is $n$, $\pi = (\pi_1,\ldots,\pi_d)^{\T}$ is a probability vector, and $g^{-1}(\theta)$ is a multinomial logit or probit link function.  Posterior computation under (\ref{eq:mnlp}) is commonly performed using data augmentation algorithms of the form 
in \eqref{eq:daupdate}.  \cite{polson2013bayesian} describe a P\'{o}lya-Gamma sampler for the multinomial logit, which is implemented in \texttt{BayesLogit}, 
while \cite{imai2005bayesian} propose a data augmentation Gibbs sampler for the multinomial probit, which is implemented in \texttt{R} package \texttt{MNP}. 

We study a synthetic data example where $y$ is a $4 \times 1$ count vector with entries adding to $n$. The first three entries of $y$ are always 1, the final entry is $n-3$, and a series of values of $n$ between $n=10$ and $n=10,000$ are considered.  Estimated values of $\Teff/T$ for the first three entries of $\theta$ for both algorithms are shown in Table \ref{tab:mnlp}. The results are similar to those for the binomial logit and probit, and are consistent across the different entries of $\theta$. It is exceedingly common for contingency tables to have many cells with small or zero entries.  Our results suggest that data augmentation algorithms should be avoided in such settings.

\begin{table}[ht]
\centering
\caption{Estimated values of $T_{\text{eff}}/T$ for the three entries of $\theta$ for multinomial logit and probit data augmentation for increasing values of $n$ 
         with data $y = (1,1,1,n-3)$. Results are based on 5,000 samples gathered after discarding 5,000 samples as burn-in.} 
\label{tab:mnlp}
\begin{tabular}{rrrrrrr}
  \hline
 & theta1 & theta2 & theta3 & theta1 & theta2 & theta3 \\ 
  \hline
n=10 & 0.0947 & 0.0974 & 0.1645 & 0.2328 & 0.2082 & 0.2638 \\ 
  n=50 & 0.0280 & 0.0292 & 0.0526 & 0.0806 & 0.0634 & 0.0548 \\ 
  n=100 & 0.0181 & 0.0198 & 0.0379 & 0.0379 & 0.0365 & 0.0399 \\ 
  n=500 & 0.0041 & 0.0033 & 0.0078 & 0.0107 & 0.0115 & 0.0150 \\ 
  n=1000 & 0.0027 & 0.0021 & 0.0065 & 0.0068 & 0.0064 & 0.0044 \\ 
  n=5000 & 0.0010 & 0.0010 & 0.0032 & 0.0013 & 0.0019 & 0.0012 \\ 
  n=10000 & 0.0002 & 0.0018 & 0.0017 & 0.0029 & 0.0008 & 0.0007 \\ 
   \hline
\end{tabular}
\end{table}

\section{Details of Data Augmentation samplers}
We provide more detail on the two data augmentation samplers considered in the main article. \cite{polson2013bayesian} introduce a data augmentation Gibbs sampler for posterior computation when $g^{-1}$ in \eqref{eq:InterceptModel} is the inverse logit link. The sampler has update rule given by
\be
\omega \mid \theta &\sim \text{PG}(n,\theta) \\
\theta \mid \omega &\sim \No{(\omega+B^{-1})^{-1} \alpha}{(\omega+B^{-1})^{-1}},
\ee
where $\alpha = y-n/2$ and $\text{PG}(a,c)$ is the \emph{P\'{o}lya-Gamma} distribution with parameters $a$ and $c$. The transition kernel $\mc P(\theta;\cdot)$ given by this update has $\theta$-marginal invariant measure the posterior $\Pi(\theta \mid y)$ for the model in \eqref{eq:InterceptModel}. 

A similar data augmentation scheme exists for the case where $g^{-1}$ is the inverse probit $\Phi(\cdot)$. Initially proposed by \cite{albert1993bayesian}, the sampler has update rule
\be \label{eq:ac}
\omega \mid \theta &= \sum_{i=1}^y z_i + \sum_{i=1}^{n-y} u_i, \qquad z_i \sim \text{TN}(\theta,1;0,\infty), \quad u_i \sim \text{TN}(\theta,1;-\infty,0) \\
\theta \mid \omega &\sim \No{(n+B^{-1})^{-1} \omega}{(n+B^{-1})^{-1}},
\ee
where $\text{TN}(\mu,\tau^2;a,b)$ is the normal distribution with parameters $\mu$ and $\tau^2$ truncated to the interval $(a,b)$. The transition kernel $\mc P(\theta;\cdot)$ for $\theta$ defined by this update has $\theta$-marginal invariant distribution $\Pi(\theta \mid y)$ for the model in \eqref{eq:InterceptModel} when $g^{-1} = \Phi$. It is clear from \eqref{eq:ac} that the computational complexity per iteration scales linearly in $n$ for this algorithm. Although a recent manuscript proposes some more efficient samplers, the samplers in \cite{polson2013bayesian} for $\mathrm{PG}(n,\theta)$ also scale linearly in $n$.


\section{Proof of \eqref{IneqAccondMain} in Theorem \ref{thm:pgcond}}
First we give a lemma that is used in the main proof to bound $\Phi^{-1}(x)$ and $(\Phi^{-1}(x))^2$. 
\begin{lemma} \label{LemmaNormalTails}
Let $\Phi(\cdot)$ be the standard normal distribution function and fix $x > 0$. Then, as $n \rightarrow \infty$,
\be
\Phi^{-1}\left(\frac{x}{n} \right) &= -(2\log(n/x))^{1/2} \left( 1 - \frac{\log\left(\frac{4}{\pi} \log(n/x) \right)}{2\log(n/x)} + \bigO{\frac{1}{( \log(n/x))^{1.5}}} \right)^{1/2}  \label{eq:Phiapprox} \\
&= -(2\log(n/x))^{1/2}(1+o(1)), 
\ee
Furthermore,
\be 
\left(\Phi^{-1}\left(\frac{x}{n} \right)\right)^2 = 2 \log\left(\frac{n}{x} \right) - \log \left( 2 \log \left( \frac{n}{x} \right)\right) + \log \left(\frac{2}{\pi} \right) + \bigO{\frac{1}{\log(n/x)}}. \label{eq:Phi2approx}
\ee 
\end{lemma}
\begin{proof}
From equations 7.8.1 and 7.8.2 of \cite{NIST:DLMF}, we have for $x\ge 0$ 
\be  \label{IneqMainNormAppr}
\frac{1}{x/\sqrt{2} + (x^2/2 + 2)^{1/2}} &< e^{x^2/2} \int_{x/\sqrt{2}}^{\infty} e^{-t^2} dt \leq \frac{1}{x/\sqrt{2} + (x^2/2 + 4/\pi)^{1/2}} \\
\frac{1}{x+ (x^{2} + 4)^{1/2}} &< ( \pi/2)^{1/2} e^{x^2/2} ( 1 - \Phi(x)) \leq \frac{1}{x+ (x^2 + 8/\pi)^{1/2}} \\
\ee 
Thus, we can write 
\be
(\pi/2)^{1/2} e^{x^2/2} \{1-\Phi(x)\} = \frac{1}{x+(x^2+h(x))^{1/2}}
\ee
for some function $h(x)$ that satisfies $8/\pi \leq h(x) \leq 4$ for all $x \ge 0$, giving
\be
1-\Phi(x) &= (2/\pi)^{1/2} e^{-x^2/2}  \frac{1}{x+(x^2+h(x))^{1/2}}.
\ee
Writing $y = \Phi(x)$, so that $y>1/2$ from the original condition for the inequality, and inverting gives
\be \label{IneqNormFirstCalcH}
1-y &= (2/\pi)^{1/2} e^{-x^2/2}  \left[x+(x^2+h(x))^{1/2}\right]^{-1} \\
\log(1-y) &= \log(2/\pi)/2 -x^2/2 -\log\left(x+(x^2+h(x))^{1/2}\right) \\
x^2 &= -2\log(1-y) + \log(2/\pi) -2\log\left(x+(x^2+h(x))^{1/2}\right).
\ee
We now claim that for any fixed $\epsilon > 0$ and any sufficiently large $x > X(\epsilon)$, we have $[-(2 - \epsilon)\log(1-y)]^{1/2} < x < [-(2 + \epsilon)\log(1-y)]^{1/2}$. To see this, recall that by Inequality \eqref{IneqMainNormAppr}, for any fixed $\epsilon > 0$,
\be
(2/\pi)^{1/2} e^{-(1+\epsilon) x^2/2} \le 1-\Phi(x) \le (2/\pi)^{1/2} e^{-(1-\epsilon) x^2/2}
\ee
for all sufficiently large $x$. Substituting this bound into \eqref{IneqNormFirstCalcH} we obtain 
\be
x^2 &= -2\log(1-y) + \log(2/\pi) - \log(-2\log(1-y)) + \bigO{\frac{1}{-\log(1-y)}},
\ee
for $1-y<1/2$ which gives 
\be
x &= \pm \left[-2\log(1-y) + \log(2/\pi) - \log(-2\log(1-y)) + \bigO{\frac{1}{-\log(1-y)}}\right]^{1/2}.
\ee
To get the result for arguments $1-y<1/2$, we take the negative solution, giving 
\be
x 
&= - (-2\log(1-y))^{1/2} \left[ 1 + \frac{\log\left(-(4/\pi) \log(1-y) \right)}{2\log(1-y)} + \bigO{\frac{1}{[-\log(1-y)]^{1.5}}} \right]^{1/2}.
\ee
Also since $(1+o(1))^{1/2} = 1 + [(1+o(1))^{1/2}-1] = 1 + o(1)$,
\be
x &= - (-2\log(1-y))^{1/2} (1+o(1)), 
\ee

Setting $1-y = x/n$ for $x/n<1/2$ -- the region where $\Phi^{-1}(x/n) < 0$ -- we have
\be
(\Phi^{-1}(x/n))^2 &= 2\log(n/x) + \log(2/\pi) - \log(2\log(n/x)) + \bigO{\frac{1}{\log(n/x)}} \\
\Phi^{-1}(x/n) &= -(2\log(n/x))^{1/2} \left( 1 - \frac{\log\left((4/\pi) \log(n/x) \right)}{2\log(n/x)} + \bigO{\frac{1}{[\log(n/x)]^{1.5}}} \right)^{1/2} \\
&= -(2\log(n/x))^{1/2}(1+o(1)),
\ee
completing the proof. 
\end{proof}

\noindent \textbf{Proof of main result} \\
The main result is proved in four steps; the rationale for each step is outlined in \S \ref{sec:verifycor}:
\begin{enumerate}[(a)]
\item Obtain bounds on quantities that will appear in steps (b) through (d);
\item Find an interval $I'(n)$ outside of which the posterior is negligible, in the sense of integrating to $o(1)$, and find an upper bound for its width;
\item Find an interval $I(n) \subset I'(n)$ containing the posterior mode $\widehat{\theta}$ on which the posterior density ratio is bounded below by a constant and show a lower bound on its width; and
\item Show a concentration inequality for $|\theta_t-\theta_{t+1}|$ when $\theta_t \in I(n)$.
\end{enumerate}

\noindent \textbf{Part (a) : obtaining additional bounds}
Recall that the Albert and Chib sampler has the update rule given by sampling $\omega_{t+1} \mid y, n, \theta_t$ then sampling $\theta_{t+1} \mid \omega_{t+1}$ from a Gaussian. $\omega_{t+1}$ is the sum of $n-y$ independent Gaussians truncated below by zero and $y$ independent Gaussians truncated above by zero; here we always have $y=1$. Then, the expectation and variance of $\omega_{t+1}$ given $\theta_t$ are 
\be
\bb E(\omega_{t+1} \mid \theta_t, n, y) &= (n-1) \left[ \theta_t - \frac{\phi(\theta_t)}{1-\Phi(\theta_t)} \right] + \left[ \theta_t + \frac{\phi(\theta_t)}{\Phi(\theta_t)} \right] \\
&= n\theta_t - (n-1) \frac{\phi(\theta_t)}{1-\Phi(\theta_t)} + \frac{\phi(\theta_t)}{\Phi(\theta_t)}  \\
\mbox{var}(\omega_{t+1} \mid \theta_t, n, y) &= v_t = n +  (n-1) \left[\theta_t \frac{\phi(\theta_t)}{1-\Phi(\theta_t)}-\left( \frac{\phi(\theta_t)}{1-\Phi(\theta_t)} \right)^2 \right] - \theta_t \frac{\phi(\theta_t)}{\Phi(\theta_t)} -\frac{\phi(\theta_t)^2}{\Phi(\theta_t)^2}.
\ee

We now compute the posterior mode $\hat{\theta}$. We begin by reparameterizing our problem by the one-to-one transformation $\theta = \Phi^{-1}(x/n)$. We will compute $\hat{x}$, the posterior mode under this transformation, and then use this to compute the mode $\hat{\theta}$ on the original scale by the equation $\hat{\theta} = \Phi^{-1}(\hat{x}/n)$. 

We will require an approximation to $\phi_B(\Phi^{-1}(x/n))$, where $\phi_B$ is the density of $N(0,B)$. Using \eqref{eq:Phi2approx}, 
\be 
\phi_B\left( \Phi^{-1}\left( \frac{x}{n} \right) \right) &= \frac{1}{(2 \pi B)^{1/2}} \exp \left[-\frac{2 \log (n/x)}{2B} + \frac{1}{2 B} \log \left( 2 \log \left( \frac{n}{x} \right) \right) -\frac{1}{2B} \log \left( 2/\pi \right) + o(1) \right] \\
&= \frac{1}{(2 \pi B)^{1/2}} \left( \frac{x}{n} \right)^{1/B} \left([2 \log(n/x)]^{1/2}\right)^{1/B} \left(\pi/2\right)^{1/2B} \exp(o(1)). \label{eq:phiB}
\ee

The posterior density when $y=1$ is proportional to
\begin{align*}
p(\theta | n,y) \propto n \Phi(\theta) (1-\Phi(\theta))^{n-1} \phi_B(\theta).
\end{align*}
Under our reparameterization, 
\be
\log p(x \mid n,y) &\propto \log x + (n-1) \log \left(1-\frac{x}{n} \right) -\frac{(\Phi^{-1}(x/n))^2}{2B}.
\ee
Differentiating to find the mode,
\be
\frac{\partial}{\partial x} \log p(x \mid n,y) 
&= \frac{1}{x} - \frac{n-1}{n-x} - \frac{(2 \pi)^{1/2}}{Bn} \exp\left( \frac{\Phi^{-1}(x/n)^2}{2} \right) \Phi^{-1}(x/n).
\ee
Using \eqref{eq:Phi2approx} and \eqref{eq:Phiapprox}, we have
\be
\frac{\partial}{\partial x} \log p(x \mid n,y) &= \frac{1}{x} - \frac{n-1}{n-x} - \frac{(2\pi)^{1/2}}{Bn} \left( \frac{n}{x} \right) (2 \log(n/x))^{-1/2} (2/\pi)^{1/2} \\
&\times \exp(o(1)) \left[ -(2 \log(n/x))^{1/2} (1 + o(1)) \right] \\
&=  \frac{1}{x} - \frac{n-1}{n-x} + \frac{2}{Bx} \left( 1 + o(1) \right),
\ee
so in the limit as $n \to \infty$ the posterior mode is 
\be \label{EqPostMode}
\widehat{x} = \frac{n(B+2 + o(1))}{Bn+2}.
\ee
In particular, for large enough $n$, $\widehat{x}/n$ is in the interval 
$$\frac{\widehat{x}}{n} \in \left[\frac{B+2}{2(Bn+2)} ,\frac{2(B+2)}{Bn+2}\right].$$ 

\noindent \textbf{Part (b) : find an interval $I'(n)$ outside of which the posterior is negligible}
We now implement the first part of the approach to showing conditions \eqref{IneqCor1} and \eqref{IneqCor2}. As described in \S \ref{sec:verifycor}, we show an interval $I'(n)$ outside of which the posterior is negligible, that is, integrates to $o(1)$. Fix $ C > 2$ and consider the interval $I'(n)=[\Phi^{-1}(n^{-C^2}),\Phi^{-1}(1-n^{-C^2})]$. 
First we bound the width of this interval and the size of the increments $|\Phi^{-1}(1-n^{-(C+1)^2}) - \Phi^{-1}(1-n^{-C^2})|$. Bounding the width from above is necessary for showing condition \eqref{IneqCor1}, and bounding the size of the increments is necessary to show that the posterior integrates to $o(1)$ outside this interval.  From \eqref{eq:Phiapprox}:
\be
\Phi^{-1}(1-n^{-C^2}) &= -\Phi^{-1}(n^{-C^2}) =  (2\log(n^{C^2}))^{1/2} (1+o(1))  \\
&= C (2 \log(n))^{1/2} (1+o(1)).
\ee
So then 
\be
|\Phi^{-1}(n^{-C^2})-\Phi^{-1}(1-n^{-C^2})| &= 2 C (2 \log(n))^{1/2} (1+o(1)) \label{eq:negregsize}
\ee
and
\be
|\Phi^{-1}(n^{-(C+1)^2})-\Phi^{-1}(1-n^{-C^2})| &= (2 \log(n))^{1/2} (1+o(1)).
\ee
Now we bound the posterior density $\Phi^{-1}(1-n^{-C^2})$, which will be used to bound the integral of the posterior on the complement of $I'(n)$
\begin{align*}
p( \theta \mid y=1) &= n (1-n^{-C^2}) \left( 1- (1-n^{-C^2}) \right)^{n-1} \phi_B(\Phi^{-1}(1-n^{-C^2})) \\
&\le n^{-C^2 n} n^{C^2+1} (2 \pi B)^{-1/2}. 
\end{align*}

We have with $p_1(\theta) = p(\theta \mid y=1)$
\be
\int_{\Phi^{-1}(1-n^{-C^2})}^{\infty} p_1(\theta) d\theta &\le (2\pi B)^{-1/2} \sum_{C=2}^{\infty} (2 \log (n))^{1/2}(1 + o(1)) n^{C^2+1-C^2n} = o(1).   
\ee
So the posterior measure of the part of $\{I'(n)\}^c$ that contains values of $\theta$ greater than those in $I'(n)$ is $o(1)$.

Now we take the same approach to show this for the part of $\{I'(n)\}^c$ consisting of values of $\theta$ less than those in $I'(n)$. We have that $p(\theta \mid y=1)$ for $\theta = \Phi^{-1}(n^{-C^2})$ satisfies
\be
p(\theta \mid y=1) &= n (n^{-C^2}) (1-n^{-C^2})^{n-1} (2 \pi B)^{-1/2} \\
&\le n^{1-C^2} e^{-n^{-(C^2-1)}} (1-n^{-C^2})^{-1} (2 \pi B)^{-1/2} \\
&\le n^{1-C^2} e^{-1} (4/3) (2 \pi B)^{-1/2},
\ee
when $n\ge 2$. So then
\be
\int_{-\infty}^{\Phi^{-1}(n^{-C^2})} p_0(\theta) d\theta &\le (4/3) (2 \pi B)^{-1/2} \sum_{C=2}^{\infty} (2 \log (n))^{1/2}(1 + o(1)) n^{1-C^2} = o(1).
\ee
We conclude
\be 
\int_{\{I'(n)\}^c} p(\theta \mid y=1) d\theta = o(1) \label{eq:negreg}
\ee
for $I'(n)=[\Phi^{-1}(n^{-C^2}),\Phi^{-1}(1-n^{-C^2})]$ with $C > 2$, so the posterior is negligible outside an interval of length $\bigO{\sqrt{\log(n)}}$ based on \eqref{eq:negregsize}.

\noindent \textbf{Part (c): Find an interval $I(n) \subset I'(n)$ containing the mode on which the posterior is almost constant} \\
We now do the second step outlined in \S \ref{sec:verifycor} to show \eqref{IneqCor1} and \eqref{IneqCor2}. We show an interval $I(n)$ containing the posterior mode on which the posterior is bounded below by a constant for all large $n$. Again fix a constant $2<C < \infty$. 
We now show that for $n>N(C)$ sufficiently large, where the function $N(C)$ depends only on $C$, the posterior is almost constant on the interval 
$$I(n)=\left[\Phi^{-1}\left(\frac{B+2}{C (Bn+2)} \right),\Phi^{-1}\left( \frac{C(B+2)}{Bn+2} \right) \right].$$ 
As shown in Equality \eqref{EqPostMode}, this interval includes the posterior mode for all large enough $n$. This interval has width $\Omega \left( (\log n)^{-1/2} \right)$. To see this, put $q(n) = (B+2)^{-1}(Bn + 2)$, then
\be
\left| I(n) \right| &= [2 \log(C q(n)) ]^{1/2} \left[ 1- \frac{\log \left[ (4/\pi) \log (C q(n)) \right]}{2 \log(Cq(n))} + \bigO{ [\log(q(n))]^{-1.5} }\right]^{1/2} \\
&-[2 \log( q(n)/C) ]^{1/2} \left[ 1- \frac{\log \left[ (4/\pi) \log (q(n)/C) \right]}{2 \log(q(n)/C)} + \bigO{[\log(q(n))]^{-1.5} } \right]^{1/2} \\
&\equiv f_1(n,C)-f_2(n,C),
\ee
where $|I(n)|$ is the width of $I(n)$. Now multiply the right side by $f_1(n,C)+f_2(n,C)$ to get
\begin{align*}
&2 \log(C q(n)) -  \log \left[(4/\pi) \log (C q(n)) \right] + \bigO{[\log(q(n))]^{-1/2} } \\
&-2 \log(q(n)/C) + \log \left[ (4/\pi) \log (q(n)/C) \right] + \bigO{ [\log(q(n))]^{-1/2} }  \\
&= 4\log(C) - 2 \log(\log(C)) + o(1).
\end{align*}
Since 
\be
f_1(n,C) + f_2(n,C) = \bigO{(\log n)^{1/2}} 
\ee
we get that 
\be
\left| I(n) \right| = \Omega \left( (\log n)^{-1/2} \right). \label{eq:intwidth}
\ee

Recall the posterior mode is $\widehat{\theta} = \Phi^{-1}((B+2+o(1))/(Bn+2))$, which is contained in $I(n)$ for sufficiently large $n$. Set $\theta_0 = \Phi^{-1}(C(B+2)/(Bn+2))$. We will bound the ratio of the posterior densities on the interval $I_1(n) = [\widehat{\theta},\theta_0]$, which is a subset of our interval $I(n)$. Repeatedly applying Lemma \ref{LemmaNormalTails}, we have
\be
\frac{p\left( y=1 \mid \widehat{\theta} \right)}{p\left( y=1 \mid \theta_0 \right)} &= \frac{n \left(\frac{B+2+o(1)}{Bn +2} \right) \left(1-\left( \frac{B+2+o(1)}{Bn +2} \right)\right)^{n-1} \phi_B\left( \frac{B+2+o(1)}{Bn +2} \right)}{n \left(\frac{C(B+2)}{Bn +2} \right) \left[1-\left( \frac{C(B+2)}{Bn +2} \right)\right]^{n-1} \phi_B\left( \frac{C(B+2)}{Bn +2} \right)} \\
&= \frac{n \left(\frac{B+2}{Bn +2} \right) \left[1-\left( \frac{B+2}{Bn +2} \right)\right]^{n-1} \phi_B\left( \frac{B+2+o(1)}{Bn +2} \right)}{n \left(\frac{C(B+2)}{Bn +2} \right) \left[1-\left( \frac{C(B+2)}{Bn +2} \right)\right]^{n-1} \phi_B\left( \frac{C(B+2)}{Bn +2} \right)} + o(1) \\
&= \frac{\left(\frac{B+2}{Bn +2} \right) \left( \frac{Bn+2}{Bn-B} \right) \left(1- \frac{B+2}{Bn +2} \right)^n \phi_B\left( \frac{B+2}{Bn +2} \right)}{\left(\frac{C(B+2)}{Bn +2} \right) \left( \frac{Bn+2}{Bn-BC} \right) \left(1-\frac{C(B+2)}{Bn +2} \right)^n \phi_B\left( \frac{C(B+2)}{Bn +2} \right)} + o(1)\\
&= \frac{ (n-C) \left(1- \frac{B+2}{Bn +2} \right)^n \phi_B\left( \frac{B+2}{Bn +2} \right)}{C (n-1) \left(1-\frac{C(B+2)}{Bn +1} \right)^n \phi_B\left( \frac{C(B+2)}{Bn +2} \right)} + o(1) \\
&= \left( 1/C + o(1) \right) \left[e^{(B+2)(C-1)/(B+2/n)} + o(1) \right]  \frac{\phi_B\left( \frac{B+2}{Bn +2} \right)}{\phi_B\left( \frac{C(B+2)}{Bn +2} \right)} + o(1) \\
&= \left( 1/C + o(1) \right) \left[e^{(B+2)(C-1)/(B+o(1))} + o(1) \right] \\
&\times \frac{\left( \frac{B+2+o(1)}{Bn+2} \right)^{1/B} \left[2 \log\left( \frac{Bn+2}{B+2+o(1)} \right) \right]^{1/(2B)} e^{o(1)}}{\left( \frac{C(B+2)}{Bn+2} \right)^{1/B} \left[2 \log\left( \frac{Bn+2}{C(B+2)} \right) \right]^{1/(2B)} e^{o(1)}} +o(1)\\
&= \left( 1/C + o(1) \right) \left[e^{(B+2)(C-1)/(B+o(1))} + o(1) \right] \left(1/C + o(1)\right)^{1/B} \\
&\times \left[\frac{\log\left( \frac{Bn+2}{B+2+o(1)} \right)}{\log\left( \frac{Bn+2}{C(B+2)} \right)} \right]^{1/(2B)} e^{o(1)} + o(1) \\
&= \left( 1/C + o(1) \right) \left[e^{(B+2)(C-1)/(B+ o(1))} + o(1) \right] \left(1/C + o(1) \right)^{1/B} \\
&\times \left[\frac{\log(Bn+2)}{\log(Bn+2)-\log(C(B+2))} - o(1) \right]^{1/(2B)} e^{o(1)} + o(1).
\ee
So then 
\be
\lim_{n \to \infty} \frac{p\left( y=1 \mid \widehat{\theta} \right)}{p\left( y=1 \mid \theta_0 \right)} = \left(1/C \right)^{1+1/B} e^{(B+2)(C-1)/B},
\ee
so in particular, since the posterior is unimodal and $\theta_0$ is the endpoint of $I_1(n)$, there exists $N_0(C)<\infty$ such that $n > N_0(C)$ implies 
\be
\inf_{\theta_0 \in I_1(n)} \frac{p\left( y=1 \mid \widehat{\theta} \right)}{p\left( y=1 \mid \theta_0 \right)} > (1/2) \left( 1/C \right)^{1+1/B} e^{(B+2)(C-1)/B}.
\ee
Now define $I_2(n) = [\theta_1,\widehat{\theta}]$ with $\theta_1 = \Phi^{-1}((B+2)/[C(Bn+2)]) = \Phi^{-1}([(1/C)(B+2)]/(Bn+2))$. Then clearly, there exists $N_1(C)<\infty$ such that $n > N_1(C)$ such that 
\be
\inf_{\theta_0 \in I_2(n)} \frac{p\left( y=1 \mid \widehat{\theta} \right)}{p\left( y=1 \mid \theta_1 \right)} > (1/2) C^{1+1/B} e^{(B+2)(1/C-1)/B}.
\ee
Put $N(C) = \max(N_0(C),N_1(C))$. Then since $I(n) = I_1(n) \cup I_2(n)$, $n > N(C)$ implies
\be
\inf_{\theta_0 \in I(n)} \frac{p\left( y=1 \mid \widehat{\theta} \right)}{p\left( y=1 \mid \theta_1 \right)} > (1/2) \min \bigg( &\left(1/C\right)^{1+1/B} e^{(B+2)(C-1)/B},\\
&C^{1+1/B} e^{(B+2)(1/C-1)/B} \bigg). \label{eq:lbcons}
\ee
Combining with (\ref{eq:intwidth}), this implies the posterior density is bounded below by a constant on an interval of width $\Omega \left( [\log(n)]^{-1/2} \right)$. Parts (b) and (c) together then give $c^*(n) = \Omega((\log n)^{-1/2})$, $C^*(n) = \bigO{(\log n)^{1/2}}$, and $1-\epsilon(n) = \Omega((\log n)^{-1})$.

\noindent \textbf{Part (d) : show a concentration result for $|\theta_{t+1} - \theta_t|$ inside $I(n)$}. 
We now show a concentration inequality for $|\theta_t-\theta_{t+1}|$ for $\theta_t$ inside the interval $I(n)$. Fix a constant $2<C<\infty$. When we have $\theta_t$ inside the interval 
$$I(n) = \left[\Phi^{-1}\left(\frac{B+2}{C (Bn+2)} \right),\Phi^{-1}\left( \frac{C(B+2)}{Bn+2} \right) \right],$$ 
we can write $\theta_t = \Phi^{-1}\left( \frac{a_t (B+2)}{Bn +1} \right)$ for $a_t \in [C^{-1},C]$, which by \eqref{EqPostMode} contains the posterior mode for large enough $n$.  The term $\phi\left(\Phi^{-1}\left( \frac{a_t(B+2)}{Bn+2} \right)\right)$ will appear often. We have that
\be
\phi \left( \Phi^{-1} \left( \frac{a_t(B+2)}{Bn+2} \right)\right) &= \bigO{\frac{[2 \log (Bn+2)]^{1/2}}{Bn+2}} 
\ee
by (\ref{eq:phiB}).

The conditional mean of $\theta_{t+1} \mid \omega_{t+1}$ will be approximately $\omega_{t+1}/n$ for large $n$, so we calculate the first two moments of $\omega_{t+1}/n$ for use in the concentration argument that follows. For $\theta_t \in I(n)$ as above we have
\be
\bb E\left( \omega_{t+1}/n \mid \theta_t, y, n \right) &= \Phi^{-1}\left( \frac{a_t(B+2)}{Bn +1} \right) - \phi\left( \Phi^{-1}\left( \frac{a_t(B+2)}{Bn +1} \right) \right) \\
&\times \left(\frac{n-1}{n} \frac{1}{1-\Phi\left(\Phi^{-1}\left( \frac{a_t(B+2)}{Bn +1} \right)\right)} - \frac{1}{n\Phi\left(\Phi^{-1}\left( \frac{a_t(B+2)}{Bn +1} \right)\right)}\right) \\
&= \Phi^{-1}\left( \frac{a_t(B+2)}{Bn +1} \right) - \phi\left( \Phi^{-1}\left( \frac{a_t(B+2)}{Bn +1} \right) \right) \\
&\times  \left(\frac{n-1}{n} \frac{Bn+2}{Bn+2-a_t(B+2)} - \frac{Bn+2}{n a_t(B+2) }\right) \\
&= \Phi^{-1}\left( \frac{a_t(B+2)}{Bn +2} \right) - \bigO{\frac{(2 \log (Bn+2))^{1/2}}{Bn+2}} \bigO{1} \\
&= \Phi^{-1}\left( \frac{a_t(B+2)}{Bn +2} \right) +\bigO{n^{-1} (\log n)^{1/2}}, \label{eq:expectomega}
\ee
and
\be
\mbox{var}(\omega_{t+1}/n \mid \theta_t, y, n) &= \frac{1}{n} +  \frac{n-1}{n^2} \left[\theta_t \frac{\phi(\theta_t)}{1-\Phi(\theta_t)}-\left( \frac{\phi(\theta_t)}{1-\Phi(\theta_t)} \right)^2 \right] - \frac{\theta_t}{n^2} \frac{\phi(\theta_t)}{\Phi(\theta_t)} -\frac{\phi(\theta_t)^2}{n^2 \Phi(\theta_t)^2} \\
&= \frac{1}{n} +  \theta_t \phi(\theta_t) \left(\frac{n-1}{n^2} \frac{1}{1-\Phi(\theta_t)} - \frac{1}{n^2} \frac{1}{\Phi(\theta_t)} \right)  \\
&- \phi(\theta_t)^2 \left( \frac{n-1}{n^2(1-\Phi(\theta_t))^2} -\frac{1}{n^2 \Phi(\theta_t)^2} \right) \\
&= \frac{1}{n} +  \theta_t \phi(\theta_t) \left(\frac{n-1}{n^2} \frac{Bn+2}{Bn+2-a_t(B+2)} - \frac{1}{n^2} \frac{Bn+2}{a_t(B+2)} \right) \\
&- \phi(\theta_t)^2 \left( \frac{n-1 (Bn+2)^2}{n^2(Bn+2-a_t(B+2))^2} -\frac{(Bn+2)^2}{n^2 (a_t(B+2))^2} \right) \\
&= \frac{1}{n} +  \theta_t \phi(\theta_t) \bigO{1} - \phi(\theta_t)^2 \bigO{1} \\
&= \frac{1}{n} +  \Phi^{-1}\left( \frac{a_t(B+2)}{Bn+2} \right) \bigO{\frac{(2 \log (Bn+2))^{1/2}}{Bn+2}}  + \bigO{\frac{2 \log (Bn+2)}{(Bn+2)^2}} \\
&= \frac{1}{n} +  \left( \left(2 \log\left( \frac{Bn+2}{a_t(B+2)} \right)\right)^{1/2}  + \bigO{ \frac{ \log\left( 2 \log\left( Bn+2 \right) \right)}{(2 \log\left( Bn+2 \right))^{1/2}}} \right) \\ 
&\times \bigO{\frac{(2 \log (Bn+2))^{1/2}}{Bn+2}} + \bigO{\frac{2 \log (Bn+2)}{(Bn+2)^2}} \\
&= \frac{1}{n} +  \bigO{\frac{2 \log (Bn+2)}{Bn+2}} + \bigO{\frac{\log( 2 \log(Bn+2))}{Bn+2}} + \bigO{\frac{2 \log (Bn+2)}{(Bn+2)^2}} \\
&= \bigO{\frac{\log n}{n}} \label{eq:varomega}
\ee
Next, for $\theta_t = \Phi^{-1}\left\{ \frac{a_t(B+2)}{Bn+2} \right\}$ -- equivalently, $\theta_t \in I(n)$ -- we want to show a uniform upper bound on $\PR{|\theta_t - \theta_{t+1}| > r \zeta}$. Our strategy is to show a uniform lower bound on $\PR{|\theta_t - \theta_{t+1}| < r \zeta}$ for $\zeta > 0, r\ge 1$. By the triangle inequality,

\be
|\theta_t - \theta_{t+1}| < \left|\theta_t - \frac{\omega_{t+1}}{n} \right| + \left|\frac{\omega_{t+1}}{n} - \theta_{t+1} \right|.
\ee
It follows that,
\be
\PR{|\theta_t - \theta_{t+1}| < r \zeta} &\ge \PR{ \left|\theta_t - \frac{\omega_{t+1}}{n} \right| < \frac{r \zeta}{2}, \left|\frac{\omega_{t+1}}{n} - \theta_{t+1} \right| < \frac{r \zeta}{2}} \\
&\ge \PR{ \left|\theta_t - \frac{\omega_{t+1}}{n} \right| < \frac{r \zeta}{2}} \PR{ \left|\omega_{t+1}/n - \theta_{t+1} \right| < \frac{r \zeta}{2} \quad \bigg| \quad \left|\theta_t - \omega_{t+1}/n \right| < \frac{r \zeta}{2}}. 
\ee
Since $\theta_t = \Phi^{-1}\left( \frac{a_t(B+2)}{Bn+2} \right)$, the first term on the right side is
\be
\pr\left[ \left| \Phi^{-1}\left( \frac{a_t(B+2)}{Bn+2} \right) - \frac{\omega_{t+1}}{n} \right| < \frac{r \zeta}{2} \right].
\ee
By \eqref{eq:varomega} and \eqref{eq:expectomega}, there exists a constant $1<A<\infty$ and an $N_0 < \infty$ such that $n>N_0$ implies $\mbox{var}(\omega_{t+1}/n \mid \theta_t, y, n) < A^2 (\log n) n^{-1}$, and 
\be
\delta_{\omega}(n) = \left|\bb E\left( \omega_{t+1}/n \mid \theta_t, y, n \right) -  \Phi^{-1}\left( \frac{a_t(B+2)}{Bn +2} \right) \right| < 2 A (\log n)^{1/2} n^{-1/2}
\ee
Putting $\zeta = 8 A n^{-1/2} (\log n)^{1/2}$ and recognizing that the distribution of $\omega_{t+1} \mid \theta_t$ is sub-Gaussian, we have, applying \eqref{eq:varomega} and \eqref{eq:expectomega}
\be
\pr \left[ \left| \frac{\omega_{t+1}}{n} -  \Phi^{-1}\left( \frac{a_t(B+2)}{Bn +1} \right) \right| > \frac{r 8 A (\log n)^{1/2}}{4n^{1/2}} + \delta_{\omega}(n) \right] \le e^{-2 r^2}  \label{eq:expconc1} \\
\pr \left[ \left| \frac{\omega_{t+1}}{n} -  \Phi^{-1}\left( \frac{a_t(B+2)}{Bn +1} \right) \right| > \frac{r 8 A (\log n)^{1/2}}{2n^{1/2}} \right] \le e^{-2 r^2}
\ee

For the second term, recall
\be
\theta_{t+1} \mid \omega_{t+1}, n &\sim \mathrm{No}\left( (n+B^{-1})^{-1} \omega_{t+1}, (n+B^{-1})^{-1}\right) \\
&\sim \mathrm{No}\left( \frac{n}{(n+B^{-1})} \frac{\omega_{t+1}}{n}, (n+B^{-1})^{-1} \right).
\ee
So then there exists $N_1<\infty$ depending only on $A,B$ such that for all $n>N_1$, the following holds using a Gaussian tail bound, 
\be
\PR{ \left|\omega_{t+1}/n - \theta_{t+1} \right| > \frac{r 8 A(\log n)^{1/2}}{4n^{1/2}} + \left| \frac{\omega_{t+1}}{n}-\frac{\omega_{t+1}}{n+B^{-1}} \right|} &\le e^{-r^2 \log (n)} \\
\PR{ \left|\omega_{t+1}/n - \theta_{t+1} \right| > \frac{r 8 A(\log n)^{1/2}}{4 n^{1/2}} + \left| \frac{(n+B^{-1})\omega_{t+1}-n\omega_{t+1}}{n(n+B^{-1})} \right|} &\le e^{-r^2 \log (n)} \\
\PR{ \left|\omega_{t+1}/n - \theta_{t+1} \right| > \frac{r 8 A(\log n)^{1/2}}{4 n^{1/2}} + \left| \frac{\omega_{t+1}}{Bn^2 + n} \right|} &\le e^{-r^2 \log (n)}.
\ee
Conditional on $\left| \Phi^{-1}\left( \frac{a_t(B+2)}{Bn+2} \right) - \frac{\omega_{t+1}}{n} \right| < \frac{r 8 A (\log n)^{1/2}}{2 n^{1/2}}$ we have
\be
\left| \frac{\omega_{t+1}}{n} \right| &< \left[2 \log \left( \frac{Bn+2}{a_t(B+2)} \right)\right]^{1/2} (1+o(1)) + \frac{r 8 A(\log n)^{1/2}}{2 n^{1/2}}, \\
&< \left[2 \log \left( \frac{C(Bn+2)}{(B+2)} \right)\right]^{1/2} (1+o(1)) + \frac{r 8 A(\log n)^{1/2}}{2 n^{1/2}},
\ee
where the second line followed since $a_t \in [1/C,C]$. So there exists $C_0 < \infty$ and a function $N_0(r)$ depending only on $r$ and $A<\infty$ such that for every $r \ge 1$, $n>N_0(r)$ implies
\be
|\omega_{t+1}| &< C_0 n (\log n)^{1/2} \\
\left| \frac{\omega_{t+1}}{Bn^2 + n} \right| &< \frac{C_0 (\log n)^{1/2}}{Bn}.
\ee
So then for any $r$ there exists $N_1(r) < \infty$ depending on $r$ and $B$ such that $n>\max(N_1(r),N_0(r))=N_{\max}(r)$ implies 
\be
\left| \frac{\omega_{t+1}}{Bn^2 + n} \right| &< \frac{r 8 A(\log n)^{1/2}}{4 n^{1/2}} \label{eq:zbound}
\ee
Since $r 8 A (\log n)^{1/2}(4 n^{1/2})^{-1}$ is increasing in $r$ and we have $r \ge 1$, choose $N_{\max} = N_{\max}(1)$, so that $n>N_{\max}$ implies \eqref{eq:zbound} uniformly over $r$. Then for all $n>\max(N_{\max},N_1)$, we have
\be
\PR{ \left|\omega_{t+1}/n - \theta_{t+1} \right| > \frac{r 8 A (\log n)^{1/2}}{2 n^{1/2}} \,\bigg|\, |\theta_t - \omega_{t+1}/n| < \frac{r \zeta}{2} } &\le e^{-r^2 \log (n)} \label{eq:expconc2}.
\ee
uniformly over $r \ge 1$. 

Putting together \eqref{eq:expconc1} and \eqref{eq:expconc2} we have for all $n > \max(N_{\max}, N_1, N_0)$, that $\theta_t \in I(n)$ implies
\be
\PR{|\theta_t - \theta_{t+1}| < r \frac{8 A(\log n)^{1/2}}{n^{1/2}}} \ge \left(1-e^{-r^2 \log (n)}\right)\left(1-e^{-2r^2}\right) \label{eq:rineq}
\ee
so 
\be
\pr \left( \left| \theta_{t+1} - \theta_t \right| > \frac{r 8 A (\log n)^{1/2}}{n^{1/2}} \right) &\leq 1-\left(1-e^{-r^2 \log (n)}\right)\left(1-e^{-2r^2 }\right) \\
&\le e^{-r^2 \log (n)} + e^{-2r^2} - e^{-\{r^2 \log (n)+2r^2\}} \\
&\le e^{-r^2 \log (n)} \left(1 - e^{-2r^2}\right) + e^{-2r^2}. 
\ee
For $r \ge 1$, the term $1 - e^{-2r^2}$ is bounded above by $1$, and $e^{-2r^2} < r^{-2}$. So then 
\be
\pr \left( \left| \theta_{t+1} - \theta_t \right| > \frac{r 8 A (\log n)^{1/2}}{n^{1/2}} \right) &\leq \bigO{n^{-1}} + r^{-2}, 
\ee
uniformly over $r$, since $e^{-r^2 \log n} = \bigO{n^{-1}}$ for $r \ge 1$.

Since the posterior is negligible outside a region of width $\bigO{(\log n)^{1/2}}$ by \eqref{eq:negregsize} and is almost constant on an interval of width $\Omega\left( (\log n)^{-1/2} \right)$ by \eqref{eq:intwidth} and \eqref{eq:lbcons}, we have $1-\epsilon(n) = \Omega\left( (\log n)^{-1} \right)$; $\zeta(n) = \bigO{n^{-1/2} (\log n)^{1/2}}$ from \eqref{eq:rineq}, $\gamma = \bigO{n^{-1}}$ from \eqref{eq:rineq} and $c^*(n) = \Omega( (\log n)^{-1/2}), C^*(n) = \bigO{(\log n)^{1/2}}$. This gives 
\be
\kappa(\mc P) = \bigO{(\log n)^{2.5} n^{-1/2}} + \bigO{n^{-1} (\log n)^2} = \bigO{(\log n)^{2.5} n^{-1/2}} 
\ee
by Corollary \ref{CorCondIneq}.

Finally, we prove Inequality \eqref{IneqAccondWarm}. Combining inequalities \eqref{EqPostMode} and \eqref{eq:lbcons} with Lemma \ref{LemmaNormalTails}, we have shown that the mode is contained within an interval of length $\Omega((\log n)^{-1/2})$ for which the density is $\Omega((\log n)^{-1/2})$. Combining inequality \eqref{eq:negreg} with Lemma \ref{LemmaNormalTails}, we have shown that the posterior distribution is negligible outside of an interval of length $\bigO{(\log n)^{1/2}}$. Inequality \eqref{IneqAccondWarm} follows immediately.

\FloatBarrier
\bibliographystyle{apalike}
\bibliography{lvmixing}
\end{document}